\documentclass[11pt]{amsart}
\usepackage{amscd,amssymb,float}

\newtheorem{theorem}{Theorem}[section]
\newtheorem{lemma}[theorem]{Lemma}

\newtheorem{proposition}[theorem]{Proposition}

\newtheorem{conjecture}[theorem]{Conjecture}

\renewcommand{\leq}{\leqslant}
\renewcommand{\geq}{\geqslant}
\newtheorem{assumptions}[theorem]{Assumptions}
\theoremstyle{definition}

\theoremstyle{definition}
\newtheorem{remark}[theorem]{Remark}
\numberwithin{equation}{section}

\newcommand{\ve}{\varepsilon}
\newcommand{\vp}{\varphi}
\newcommand{\mn}{\sqrt{-1}}
\newcommand{\ov}[1]{\overline{#1}}
\newcommand{\de}{\partial}
\newcommand{\db}{\overline{\partial}}
\newcommand{\ddbar}{\sqrt{-1} \partial \overline{\partial}}
\newcommand{\ti}[1]{\tilde{#1}}

\newcommand{\tr}[2]{\textrm{tr}_{#1} #2}
\DeclareMathOperator{\Ric}{Ric} 

\newcommand{\ZZ} {\mathbb{Z}}

\newcommand{\RR} {\mathbb{R}}
\newcommand{\CC} {\mathbb{C}}
\newcommand {\shP}  {\mathcal{P}}
\newcommand {\shL}  {\mathcal{L}}
\newcommand {\shC}  {\mathcal{C}}
\newcommand {\shK}  {\mathcal{K}}
\newcommand {\nor}  {{\operatorname{nor}}}
\renewcommand{\O}  {\mathcal{O}}
\renewcommand{\Re}  {\operatorname{Re}}
\renewcommand{\Im}  {\operatorname{Im}}
\newcommand{\Int} {\operatorname{int}}
\newcommand{\PP} {\mathbb{P}}
\newcommand {\Diff}  {\operatorname{Diff}}
\newcommand {\Per} {\operatorname{Per}}

\numberwithin{equation}{section} \numberwithin{figure}{section}

\begin{document}
\title{Collapsing of Abelian Fibred Calabi-Yau Manifolds}
\author[M. Gross]{Mark Gross$^{*}$}
\thanks{$^{*}$Supported in part by NSF grants  DMS-0805328 and DMS-1105871.}
\address{Mathematics Department, University of California San Diego, 9500 Gilman Drive \#0112, La Jolla, CA 92093}
\email{mgross@math.ucsd.edu}
\author[V. Tosatti]{Valentino Tosatti$^{\dagger}$}
\thanks{$^{\dagger}$Supported in part by NSF grant
DMS-1005457.}
 \address{Department of Mathematics, Columbia University, 2900 Broadway, New York, NY 10027}
  \email{tosatti@math.columbia.edu}
  \author[Y. Zhang]{Yuguang Zhang$^{\ddagger}$}
  \thanks{$^{\ddagger}$Supported in part  by  NSFC-10901111.}
\address{Mathematics Department, Capital Normal University, Beijing 100048, P.R. China}
\email{yuguangzhang76@yahoo.com}
\begin{abstract}
We study the collapsing behaviour of Ricci--flat K\"ahler metrics on
a projective Calabi-Yau manifold which admits an abelian fibration,
when the volume of the fibers approaches zero. We show that away
from the critical locus of the fibration the metrics collapse with
locally bounded curvature, and along the fibers the rescaled metrics
become flat in the limit. The limit metric on the base minus the
critical locus is locally isometric to an open dense subset of any
Gromov-Hausdorff limit space of the Ricci-flat metrics. We then apply these results to study metric
degenerations of families of polarized hyperk\"ahler manifolds in
the large complex structure limit. In this setting we prove an
analog of a result of Gross-Wilson for $K3$ surfaces, which is
motivated by the Strominger-Yau-Zaslow picture of mirror symmetry.
\end{abstract}

\maketitle
\section{Introduction}\label{sect0}
In this paper, a Calabi-Yau manifold $M$ is a compact K\"ahler manifold
with vanishing first Chern class $c_1(M)=0$ in $H^2(M,\mathbb{R})$.
A fundamental theorem of Yau \cite{Ya1} says that on $M$ there
exists a unique Ricci--flat K\"ahler metric in each K\"ahler class.
If we move the K\"ahler class towards a limit class on the boundary
of the K\"ahler cone, we get a family of Ricci--flat K\"ahler
metrics which degenerates in the limit. The general question of
understanding the geometric behaviour of these metrics was raised by
Yau \cite{yau2, yau3}, Wilson \cite{W} and others, and much work has
been devoted to it, see for example \cite{GW, RoZ, RZ, ST, To, To1}
and references therein. In this paper, we study  metric
degenerations of   Ricci--flat K\"ahler metrics whose K\"ahler
classes approach  semi-ample non-big  classes.

The first useful observation is that  the diameters of a family of
Ricci--flat K\"ahler metrics $\ti{\omega}_{t}$, $t\in (0,1]$,  on a
Calabi-Yau manifold $M$ are uniformly bounded if their K\"ahler
classes $[\ti{\omega}_{t}]$ tend to a limit class $\alpha$ on the
boundary of the K\"ahler cone when $t\rightarrow 0$ \cite{To, ZT}.
Another special feature of the K\"ahler case is that the volume of
the Ricci--flat metrics can be computed cohomologically, and to
determine whether it will approach zero or stay bounded away from
it, it is enough to calculate the self-intersection $\alpha^{n}$ where
$n=\dim_{\mathbb{C}}M$. If $\alpha^{n}$  is strictly positive, then it was proved by the second-named author \cite{To} that the Ricci--flat metrics do not collapse, (i.e., there is a constant $\upsilon >0$ independent of $t$ such that each $\ti{\omega}_{t}$ has a unit radius metric ball with volume bigger than $\upsilon$), and in fact converge smoothly away from a subvariety. If $\alpha^{n}$ is zero, then the total volume of the Ricci--flat metrics approaches zero, so one expects to have collapsing to a lower-dimensional space. This was shown to be the case for elliptically fibered $K3$ surfaces by Gross-Wilson \cite{GW}, and later the second-named author considered the higher dimensional case when the Calabi-Yau manifold $M$ admits a holomorphic fibration to a lower-dimensional K\"ahler space, and
the limit class is the pullback of a K\"ahler class \cite{To1}. The
first goal of the present paper is to improve the convergence result
in \cite{To1}.

Let us now describe our first result in detail. Let $(M,\omega_M)$ be a compact Calabi-Yau $n$-manifold which
admits a holomorphic map $f:M\to Z$ where $(Z,\omega_Z)$ is a
compact K\"ahler manifold. Thanks to Yau's theorem, we can assume
that $\omega_M$ is Ricci--flat. Denote by $N=f(M)$ the image of $f$,
and assume that $N$ is an irreducible normal subvariety of $Z$ with
dimension $m$, $0<m<n$, and that the map $f:M\to N$ has connected
fibers. Define $\omega_0:=f^*\omega_Z$, which is a smooth
nonnegative real $(1,1)$-form on $M$ with cohomology class on
the boundary of the K\"ahler cone, and let
$\omega_N$ be the restriction of $\omega_Z$ to the regular part of $N$.
For example, one can take either $Z=N$ (if $N$ is smooth), or
$Z=\mathbb{CP}^k$ (if $N$ is an algebraic variety). This second case
arises if we have a line bundle $L\to M$ which is semiample
(some power is globally generated) and of Iitaka dimension $m<n$, so $L$
is not big.

In general, given a map $f:M\to N$ as above, there is a proper
analytic subvariety $S\subset M$ such that $N\backslash f(S)$ is
smooth and $f:M\backslash S\to N\backslash f(S)$ is a smooth
submersion (the set $f(S)$ is exactly the image of the subset of $M$ 
where the differential $df$ does
not have full rank $m$). For any $y\in N\backslash f(S)$ the fiber
$M_y=f^{-1}(y)$ is a smooth Calabi-Yau manifold of dimension $n-m$,
and it is equipped with the K\"ahler metric $\omega_M|_{M_y}$. The
volume of the fibers $\int_{M_y}(\omega_M|_{M_y})^{n-m}$ is a
homological constant which we can assume to be $1$.
Consider the K\"ahler metrics on $M$ given by
$\omega_t=\omega_0+t\omega_M$, with $0<t\leq 1$, and call
$\ti{\omega}_t=\omega_t+\ddbar\vp_t$ the unique Ricci--flat K\"ahler
metric on $M$ cohomologous to $\omega_t$, with potentials normalized
by $\sup_M\vp_t=0$. They satisfy a family of complex
Monge-Amp\`ere equations
\begin{equation}\label{ma}
\ti{\omega}_t^n=(\omega_t+\ddbar\vp_t)^n=c_t t^{n-m} \omega_M^n,
\end{equation}
where $c_t$ is a constant that has a positive limit as $t\to 0$ (see
\eqref{normaliz}). A general $C^0$ estimate $\|\vp_t\|_{C^0}\leq C$
(independent of $t>0$) for such equations was proved by
Demailly-Pali \cite{DP} and Eyssidieux-Guedj-Zeriahi \cite{EGZ2},
generalizing work of Ko\l odziej \cite{Ko}. In the case under
consideration, much more is true: the second-named author's work
\cite{To1} shows that there exists a smooth function $\vp$ on
$N\backslash f(S)$ so that as $t$ goes to zero we have $\vp_t\to
\vp\circ f$ in $C^{1,\alpha}_{loc}(M\backslash S,\omega_M)$ for any
$0<\alpha<1$. Moreover $\omega=\omega_N+\ddbar\vp$ is a K\"ahler
metric on $N\backslash f(S)$ with $\Ric(\omega)=\omega_{\rm WP}$.
Here $\omega_{\rm WP}$ is the pullback of the Weil-Petersson metric
from the moduli space of polarized Calabi-Yau fibers, which has
appeared several times before in the literature \cite{Fi,GW, ST,
To1}.

We now assume that one (and hence every) fiber $M_y$ with $y\in N\backslash
f(S)$ is biholomorphic to a complex torus. This is
the case for example whenever $M$ is hyperk\"ahler. We also assume
that $M$ is projective, so we can take $[\omega_M]$ to be the first
Chern class of an ample line bundle. In this case we can improve the
above  result, thus answering Questions 4.1 and 4.2 of \cite{To2} in
our setting: 

\begin{theorem}\label{main1} If $M$ is projective and if one (and hence all) of the
fibers $M_y$ with $y\in N\backslash f(S)$ is a torus, then as $t$ approaches zero the Ricci--flat metrics $\ti{\omega}_t$ converge in $C^\infty_{\mathrm{loc}}(M\backslash S,\omega_M)$ to $f^*\omega$, where $\omega$ is a K\"ahler metric on $N\backslash f(S)$ with $\Ric(\omega)=\omega_{\rm WP}$.
Given any compact set $K\subset M\backslash S$ there is a constant
$C_K$ such that the sectional curvature of $\ti{\omega}_t$ satisfies
\begin{equation}\label{sectbound}
\sup_K |\mathrm{Sec}(\ti{\omega}_t)|\leq C_K,
\end{equation}
for all small $t>0$. Furthermore, on each torus fiber $M_y$ with
$y\in N\backslash f(S)$ we have
\begin{equation}\label{rescaled}
\frac{\ti{\omega}_t|_{M_y}}{t}\to \omega_{SF,y},
\end{equation}
where $\omega_{SF,y}$ is the unique flat metric on $M_y$
cohomologous to $\omega_M|_{M_y}$ and the convergence is smooth and
uniform as $y$ varies on a compact subset of $N\backslash f(S)$.
\end{theorem}

As remarked earlier, in the case of elliptically fibered $K3$
surfaces ($n=2, m=1$) this theorem follows from the work of
Gross-Wilson \cite{GW}. In higher dimensions, in the very special
case when $S$ is empty, the theorem (except \eqref{sectbound}) also
follows from the work of Fine \cite{Fi}.

The curvature bound (\ref{sectbound}) in Theorem \ref{main1} does
not hold if the generic fibers are not tori, as one can see for example
by taking the product of two non-flat Calabi-Yau manifolds
with the product Ricci-flat K\"ahler metric and then scaling
one factor to zero. On the other hand, we believe that the
assumption in Theorem \ref{main1} that $M$ is projective is just technical.

We now describe the second main result of the paper, which concerns the Gromov-Hausdorff limit of our manifolds.
The Gromov-Hausdorff distance $d_{GH}$ was introducted by Gromov in the
1980's \cite{G1}, and it defines a topology on the space
of isometry classes of all compact metric spaces. For two compact
metric spaces $X$ and $Y$, the Gromov-Hausdorff distance of $X$ and
$Y$ is
$$d_{GH}(X,Y)=\inf_{Z} \{d_H^Z(X,Y)\ |\ X, Y\hookrightarrow Z
 \text{ isometric embeddings}\},$$ where $Z$ is a metric space and $d_H^Z(X,Y)$
denotes the standard Hausdorff distance between $X$ and $Y$
regarded as subsets in $Z$ by the isometric embeddings (see for example \cite{G1, Ro} for more background). 
We would like to understand the Gromov-Hausdorff
convergence of $(M,\ti{\omega}_t)$  in
Theorem \ref{main1}. Since the volume of the whole manifold goes to zero,
the manifolds $(M,\ti{\omega}_t)$ are collapsing. Furthermore, from Theorem \ref{main1} we know that on a Zariski open set of $M$ the Ricci-flat metrics collapse with locally bounded curvature.
The collapsing of   Riemannian
manifolds  in the Gromov-Hausdorff
sense has been extensively studied from different viewpoints, see
for example \cite{An1, CC1, CG1, CT, CN, Fu, GW, NT, Ro, SSW} and
the reference therein. Especially,  the metric structure of collapsed limits of Riemannian  manifolds with a uniform lower bound on the Ricci curvature was  studied by Cheeger-Colding \cite{CC1} and collaborators.
Regarding the collapsed Gromov-Hausdorff limit of the Ricci--flat
metrics in Theorem \ref{main1}, we have the following result.

First of all, thanks to \cite{To,ZT} we know that the diameter of
$(M,\ti{\omega}_t)$ satisfies $\mathrm{diam}_{\ti{\omega}_t}(M)\leq D,$
for some constant $D$ and for all $t>0$. Furthermore, since
$\ti{\omega}_t\to f^*\omega$ and the base $N$ is not a point, we
also have that $\mathrm{diam}_{\ti{\omega}_t}(M)\geq D^{-1}.$ Given
any sequence $t_{k}\rightarrow  0$, Gromov's precompactness theorem
shows that a subsequence of $(M, \tilde{\omega}_{t_{k}}) $ converges
to some compact path  metric space $(X, d_{X})$ in the
Gromov-Hausdorff topology. Note that because of the upper and lower
bounds for the diameter, if we rescaled the metrics
$\tilde{\omega}_{t_{k}}$ to have diameter equal to one, the
Gromov-Hausdorff limit (modulo subsequences) would be isometric to
$(X, d_X)$ after a rescaling.

\begin{theorem}\label{main2}
In the same setting as Theorem \ref{main1}, for any such limit space
$(X, d_{X})$  there is an open dense  subset $X_{0}\subset X$ such
that $(X_{0}, d_{X})$ is locally isometric to $(N\backslash f(S),
\omega)$, i.e. there is a homeomorphism $\phi:  N\backslash f(S)
\rightarrow X_{0}$ satisfying  that, for any $y\in N  \backslash
f(S)$, there is a neighborhood $B_{y}\subset N \backslash f(S)$ of
$y$ such  that, for $y_{1}$ and $y_{2}\in B_{y}$,
  $$d_{\omega}(y_{1},y_{2})=d_{X}(\phi(y_{1}),\phi(y_{2})).$$
\end{theorem}

In fact we prove that $X\backslash X_0$ has measure zero with
respect to the renormalized limit measure of \cite{CC1}, which implies that $X_0$ is dense in $X$.
It would be interesting to prove that the metric completion of
$(N\backslash f(S), \omega)$ is isometric to $(X,d_X)$. For $K3$ surfaces this was proved by Gross-Wilson \cite{GW}.

As an application of Theorem  \ref{main1} and Theorem  \ref{main2}
we study the metric degenerations of  families  of polarized  hyperk\"ahler manifolds in the large complex structure limit. In \cite{SYZ}, Stominger, Yau and Zaslow proposed a conjecture, known as the SYZ conjecture,  about constructing the mirror manifold of a given
Calabi-Yau manifold  via  special Lagrangian fibrations. 
Later another version of the SYZ conjecture was proposed by
Gross-Wilson \cite{GW}, Kontsevich-Soibelman \cite{KS} and Todorov
via degenerations of Ricci--flat K\"ahler-Einstein metrics.  The
conjecture says that if $\{M_{t}\}$,
$t\in\Delta\backslash\{0\}\subset\mathbb{C} $, is a family of
polarized Calabi-Yau $n$-manifolds whose complex structures tend to a large complex
structure limit point
when $t\rightarrow 0$, and  $\omega_{t}$ is the Ricci--flat
K\"ahler-Einstein metric representing the polarization on $M_{t}$,
  then after rescaling $(M_{t}, \omega_{t})$ to
have diameter $1$, they collapse to a compact metric space $(X,
d_{X})$ in the  Gromov-Hausdorff sense. Furthermore, a dense open
subset $X_{0}\subset X$ is a smooth manifold of real dimension $n$,
and the codimension of $X\backslash X_{0}$ is bigger or equal to
$2$. This conjecture holds trivially for tori, and was verified for
$K3$ surfaces by Gross-Wilson in \cite{GW}.

In the third main result of this paper
we consider this conjecture for higher dimensional hyperk\"ahler
manifolds.  Let $(M,I)$ be a
compact hyperk\"ahler manifold of complex dimension $2n$ with a
Ricci--flat K\"ahler metric $\omega_I$.  We assume that there
is an ample line bundle over $M$ with the  first Chern class
$[\omega_I]$, that we have a holomorphic fibration $f:M\to N$ as
before with $N$ a projective variety, and that there is a
holomorphic section $s:N\to M$.  Under these assumptions, it is known
that $N=\mathbb{CP}^n$ \cite{Hw}, and that the smooth fibers of $f$
are complex Lagrangian tori \cite{Mats}. If we perform a
hyperk\"ahler rotation of the complex structure, the fibers become
special Lagrangian, and we are exactly in the setup of Strominger,
Yau and Zaslow \cite{SYZ}. We furthermore assume that the polarization induced
on the torus fibers is principal. In this case, the SYZ mirror
symmetry picture predicts that $M$ is mirror to itself, and that a
large complex structure limit is mirror to a large K\"ahler
structure limit. We use this as our definition of large complex
structure limit, so we have a family of polarized hyperk\"ahler structures
$(M,\check{\Omega}_s)$ with Ricci-flat K\"ahler
metric $\check{\omega}$ which approach a large complex structure
limit as $s\to \infty$. By assuming the validity of a standard
conjecture on hyperk\"ahler manifolds (Conjecture
\ref{conj2.3}), we prove:

\begin{theorem}\label{main3} In the above situation, denote $\check{M}_{s}$ the hyperk\"ahler
manifold with period $\check\Omega_s$, and $d_{s}={\rm
diam}_{\check{\omega}}(\check{M}_{s})$. Then, for any sequence
$s_{k}\rightarrow \infty$, a subsequence of  $(\check{M}_{s_{k}},
d_{s_{k }}^{-2}\check{\omega})$ converges in the Gromov-Hausdorff sense to a compact  metric space $(X,d_{X})$, which has  an open dense subset  $(X_{0},d_{X})$ locally isometric to an open non-complete smooth Riemannian manifold $(N_{0}, g)$ with $\dim_{\mathbb{R}}N_{0}=\frac{1}{2}\dim_{\mathbb{R}}M $.
\end{theorem}

This proves the conjecture of Gross-Wilson \cite{GW},
Kontsevich-Soibelman \cite{KS} and Todorov in our situation,
modulo these assumptions, except
for the statement that  $\mathrm{codim}_{\mathbb{R}}(X\backslash X_0)\geq 2$.
Again, this was proved by Gross-Wilson \cite{GW} in the case of $K3$ surfaces.\\

This paper is organized as follows. In Section \ref{sect1} we study SYZ mirrors of some hyperk\"ahler manifolds, and derive Theorem \ref{main3} as a consequence of Theorems \ref{main1} and \ref{main2}. In Section \ref{sect2} we construct semi-flat background metrics on the total space of a holomorphic torus fibration. Theorem  \ref{main1} is proved in Section \ref{sect3} while Theorem \ref{main2} is proved in Section \ref{sect4}.\\

\noindent {\bf Acknowledgements:} Most of this work was carried out
while the second-named author was visiting the Mathematical Science
Center of Tsinghua University in Beijing, which he would like to
thank for the hospitality. He is also grateful to
D.H. Phong and S.-T. Yau for their
support and encouragement, and to J. Song for many
useful discussions. Some parts of this paper were obtained
while the third-named  author's was visiting University of
California San Diego and  Institut des Hautes \'{E}tudes
Scientifiques. He would like to  thank UCSD and IH\'{E}S for
the hospitality, and he is also grateful to Professor Xiaochun Rong
for helpful discussions.

\section{Hyperk\"ahler mirror symmetry}\label{sect1}

In this section we discuss a version of mirror symmetry
for hyperk\"ahler manifolds analogous to the one used for K3
surfaces in \cite{GW}.
The situation for general hyperk\"ahler manifolds is considerably
less developed, however, and we shall have to make many assumptions
in this discussion. The goal is to explain the background of Theorem
\ref{main3} and show it follows from Theorems \ref{main1} and
\ref{main2}. 
This discussion largely represents a summary of known results, and
is completely analogous to \cite{GW}.

First we review known facts about periods of hyperk\"ahler
manifolds.
Fix $M$ a manifold of real dimension $4n$ which supports a
hyperk\"ahler manifold structure with holonomy being the full group
$Sp(n)$. Set $L=H^2(M,\ZZ)$,
$L_{\RR}:=L\otimes_{\ZZ}\RR$, $L_{\CC}:=L\otimes_{\ZZ} \CC$. Then
there is a real-valued non-degenerate quadratic form
$q_M:L\rightarrow \RR$, called the \emph{Beauville-Bogomolov form},
with the property that there is a constant $c$ such that
\[
q_M(\alpha)^n=c\int_M \alpha^{2n}
\]
for $\alpha\in L$, of signature $(+,+,+,-,\cdots,-)$. We write
$q_M(\cdot,\cdot)$ for the induced pairing, with
$q_M(\alpha,\alpha)=q_M(\alpha)$.

We can define the \emph{period domain} of $M$ to be
\[
\shP_M:=\{[\Omega]\in \PP(L_{\CC})\,|\, q_M(\Omega)=0,\quad
q_M(\Omega,\bar\Omega)>0\}.
\]
The Teichm\"uller space of $M$, ${\bf Teich}_M$, is the set of
hyperk\"ahler complex structures on $M$ modulo elements of
$\Diff_0(M)$, the diffeomorphisms of $M$ isotopic to the identity.
By the Bogomolov-Tian-Todorov theorem, this is a (non-Hausdorff)
manifold. There is a period map
\[
\Per:{\bf Teich}_M\rightarrow \shP_M
\]
taking a complex structure on $M$ to the class of the line
$H^{2,0}(M)$. Then $\Per$ is \'etale, and was proved to be
surjective by Huybrechts in \cite{Huy1}. 
While recently Verbitsky \cite{Verb1}
proved a suitably formulated global Torelli theorem, one
must keep in mind that $\Per$ is not, in general, a diffeomorphism.

Next consider a complex structure on $M$ and Ricci--flat K\"ahler
metric $\omega_I$ making $M$ hyperk\"ahler. Then a choice of a
holomorphic symplectic two-form $\Omega_I$, along with $\omega_I$,
completely determines this structure. In particular, if we write
$\Omega_I=\omega_J+\mn\omega_K$, we can normalize $\Omega_I$ so that
$q_M(\omega_I)=q_M(\omega_J)=q_M(\omega_K)$. Furthermore,
necessarily $q_M(\omega_I,\omega_J)=q_M(\omega_I,\omega_K)
=q_M(\omega_J,\omega_K)=0$. The triple $\omega_I,\omega_J,\omega_K$
is called a \emph{hyperk\"ahler triple}. It gives rise to an $S^2$
worth of complex structures compatible with the same hyperk\"ahler
metric: in particular, one has the $J$ complex structure with
holomorphic symplectic form $\Omega_J:=\omega_K+\mn\omega_I$ and
K\"ahler form $\omega_J$, and the $K$ complex structure with
holomorphic symplectic form $\Omega_K:=\omega_I+\mn\omega_J$ and
K\"ahler form $\omega_K$.\\

Now suppose that we are given a complex structure on $M$ such that there
is a fibration $f:M\rightarrow N$, with fibers being holomorphic
Lagrangian subvarieties of $M$. Suppose furthermore that $N$ is a
K\"ahler manifold. Then by results of Matsushita (see \cite{Mats}
and \cite{GHJ}, Proposition 24.8) the smooth
fibers of $f$ are complex tori and $N$ is a Fano manifold with
$b_2(N)=1$. Furthermore, if $M$ is projective of complex dimension
$2n$ then $N=\mathbb{CP}^n$ by a result of Hwang \cite{Hw}. If
$M_y$ is a fiber of $f$, then $\omega_J|_{M_y}=\omega_K|_{M_y}=0$,
from which it follows that $\Im(\Omega_K^n)|_{M_y}=0$, so that
after hyperk\"ahler rotation fibers
of $f$ are special Lagrangian.

The Strominger-Yau-Zaslow conjecture \cite{SYZ} predicts that mirror
symmetry can be explained via dualizing such a special Lagrangian
torus fibration. In a general situation, it can be hard to dualize
torus fibrations, because of singular fibres.
The case that $M$ is a K3 surface, treated in
detail in \cite{GW}, is rather special because Poincar\'e duality
gives a canonical isomorphism between a two-torus and its dual.

With some additional assumptions, a similar situation holds in the
hyperk\"ahler case. Suppose that the K\"ahler form $\omega_I$ is
integral, so that there is an ample line bundle $\shL$ on $X$ whose
first Chern class is represented by $\omega_I$. The restriction of
this line bundle to a non-singular fiber $M_y$ then induces a
polarization of some type $(d_1,\ldots,d_n)$. In particular there is
a canonical map $M_y\rightarrow M_y^{\vee}$ given by
\[
M_y\ni x\mapsto \shL|_{M_y}\otimes t_x^*\shL^{-1}|_{M_y}\in
M_y^{\vee}.
\]
Here $M_y^{\vee}$ is the dual abelian variety to $M_y$, classifying
degree zero line bundles on $M_y$, and $t_x:M_y\rightarrow M_y$ is
given by translation by $x$, which makes sense once one chooses an
origin in $M_y$. The kernel of this map is
$(\ZZ/d_1\ZZ\oplus\cdots\oplus\ZZ/d_n\ZZ)^{\oplus 2}$. In
particular, if $f$ possesses a section $s:N\rightarrow M$, and
$N_0:=N\setminus f(S)$ where $S$ is as in the introduction, then
the dual of $f^{-1}(N_0)\rightarrow N_0$ can be described as a
quotient map, given by dividing out by the kernel of the
polarization on each fiber. One can then hope that this dual
fibration can be compactified to a hyperk\"ahler manifold.

In general, if $M_y$ carries a polarization of type
$(d_1,\ldots,d_n)$, it is not difficult to check that the dual abelian variety $M_y^{\vee}$
carries a polarization of type
$(d_n/d_n,d_n/d_{n-1},\ldots,d_n/d_1)$. Thus if the polarization is not
principal, the SYZ dual manifold need not coincide with $M$: see
examples of non-principally polarized fibrations in
Example 3.8 and Remark 3.9 of
\cite{Sawon}. It is possible that Sawon's fibrations do not have
duals which are hyperk\"ahler manifolds, as a natural
compactification might be a holomorphic symplectic variety without a
holomorphic symplectic resolution of singularities.

On the other hand, if $\omega_I$ induces a \emph{principal}
polarization on each fiber $M_y$, i.e., the map $M_y\rightarrow
M_y^{\vee}$ is an isomorphism, then the SYZ dual of the fibration
$f^{-1}(N_0)\rightarrow N_0$, assuming again the existence of a
section, can be canonically identified with $f^{-1}(N_0)\rightarrow
N_0$, and thus it is natural to consider $f:M\rightarrow N$ to be a
self-dual fibration, at least at the purely topological level. In
this case, and only in this case, SYZ mirror symmetry predicts that
hyperk\"ahler manifolds are self-mirror. The idea that hyperk\"ahler
manifolds should be self-mirror was first suggested and explored by
Verbitsky in \cite{Verb2}.

In this case only, we can be more explicit about mirror symmetry. We
summarize our assumptions so far:

\begin{assumptions}
Let $M_I$ be a hyperk\"ahler manifold with $f:M_I\rightarrow N$ a
complex torus fibration, along with a section $s:N\rightarrow M_I$
and an ample line bundle $\shL$ with first Chern class represented
by a hyperk\"ahler metric $\omega_I$. We assume further the induced
polarization on the smooth fibers of $f$ is principal and that $N$
is projective.
\end{assumptions}

 Thus, with these assumptions, it is natural to assume
that mirror symmetry exchanges complex and K\"ahler moduli for the
fixed underlying space $M$. This can be described at the level of
period domains as follows.

Let $\sigma\in L_{\RR}$ be the class represented by $\omega_I$. Fix
an integral K\"ahler class $\omega_N$ on $N$, and let $E\in L$ be
represented by $f^*\omega_N$, so that $q_M(E)=0$.

\begin{lemma} In the above situation, we have
$q_M(E,\sigma)\not=0$.
\end{lemma}

\begin{proof}
By \cite{GHJ}, Exercise 23.2, we have
\[
q_M(E,\sigma)\int_M \sigma^{2n}=2q_M(\sigma)\int_M
\sigma^{2n-1}\wedge f^*\omega_N\not=0,
\]
so $q_M(E,\sigma)\not=0$.
\end{proof}

Denote by $E^{\perp}\subseteq L_{\RR}$ the orthogonal complement of
$E$ under $q_M$, and denote by $E^{\perp}/E$ the quotient space
$E^{\perp}/\RR E$. Then $q_M$ induces a quadratic form on
$E^{\perp}/E$. Let
\[
\shC(M):=\{x\in E^{\perp}/E\,|\, q_M(x)>0\},
\]
and define the \emph{complexified K\"ahler moduli space} of $M$ to
be
\[
\shK(M):=(E^{\perp}/E)\oplus i\shC(M)\subseteq
(E^{\perp}/E)\otimes\CC.
\]

We then have an isomorphism
\[
m_{E,\sigma}:\shK(M)\rightarrow \shP_M\setminus E^{\perp}
\]
via, representing an element of $(E^{\perp}/E)\otimes\CC$ by
$\alpha\in E^{\perp}\otimes\CC$,
\[
\alpha\mapsto \left[\frac{1}{q_M(E,\sigma)}\sigma+\alpha -\frac{1}{2}
\left(\frac{q_M(\sigma)}
{q_M(E,\sigma)^2}+q_M(\alpha)+2\frac{q_M(\alpha,\sigma)}{q_M(E,\sigma)}\right)E\right].
\]
Indeed, one first checks that this is independent of which
representative $\alpha$ is chosen. One then notes that the
coefficient of $E$ is chosen so that $q_M(m_{E,\sigma}(\alpha))=0$,
and $q_M(m_{E,\sigma}(\alpha),
m_{E,\sigma}(\bar\alpha))=2q_M(\Im\alpha)>0$ by assumption that
$\alpha\in \shK(M)$. Further, $m_{E,\sigma}$ is clearly injective,
since $\alpha=m_{E,\sigma}(\alpha)-\sigma/q_M(E,\sigma) \mod E$. It
is surjective, since given $[\Omega]\in \shP_M\setminus E^{\perp}$,
we can rescale $\Omega$ so that $q_M(\Omega,E)=1$, and then
$[\Omega]=m_{E,\sigma}(\Omega-\sigma/q_M(E,\sigma)\mod E)$.

We can then view the mirror map $m_{E,\sigma}$ described above as
realising mirror symmetry on the level of period domains,
defining an exchange of data
\[
(M,\Omega, {\bf B}+\mn\omega)\leftrightarrow (M,\check\Omega,
\check{\bf B} +\mn\check\omega).
\]
Here $[\Omega],[\check\Omega]\in \shP_M$, with
$q_M(E,\Omega),q_M(E,\check\Omega)\not=0$, so that we can assume
$\Omega$ and $\check\Omega$ are normalized with
$q_M(E,\Omega)=q_M(E,\check\Omega)=1$. Furthermore, ${\bf
B},\check{\bf B}\in E^{\perp}/E$ and $\omega, \check\omega\in
E^{\perp}$ satisfy $q_M(\omega,\Omega)=q_M(\check\omega,
\check\Omega)=0$ and $q_M(\omega),q_M(\check\omega)>0$. The
relationship between the two triples is that
$\check\Omega=m_{E,\sigma}({\bf B}+\mn\omega)$ and $\check{\bf B},
\check\omega$ are the unique cohomology classes satisfying the above
conditions and $\Omega=m_{E,\sigma}(\check{\bf B}+\mn\check\omega)$.
Indeed, $\check{\bf B}$ and $\check\omega$ exist, since as
$q_M(E,\Omega)=1$, we can write $\Omega=\frac{1}{q_M(E,\sigma)}\sigma+\check{\bf B}+\mn\check\omega \mod E$, and
replacing a chosen representative $\check\omega$ with
$\check\omega-(q_M(\check\omega,\sigma)/q_M(E,\sigma)-q_M(\check\omega,
{\bf B}))E$, one guarantees that $q_M(\check\Omega,\check\omega)=0$.

This mirror symmetry on the level of period domains doesn't quite
give an exact mirror symmetry on the level of moduli spaces, since
global Torelli does not in general hold,
so there might be a number of choices of complex structure on $M$
with period $[\Omega]$. In addition, $\omega$ or $\check\omega$ need
not represent a K\"ahler form except for very general choices of
complex structure.

Nevertheless, this allows us to identify a large complex structure
limit as being mirror to a large K\"ahler limit. The family
\[
\left(M,\Omega=\frac{1}{q_M(E,\sigma)}\sigma+\check{\bf
B}+\mn\check\omega \bmod E, s\omega\right),
\]
represents a large K\"ahler limit, with the K\"ahler class
moving off to infinity while the complex structure is fixed, and this is mirror to the triple
\[
\left(M,\check\Omega_s=\frac{1}{q_M(E,\sigma)}\sigma+\mn
s\omega\bmod E,\check{\bf B} +\mn\check\omega\right).
\]
If, for each $s$, we have an actual hyperk\"ahler manifold with
period $\check\Omega_s$ and K\"ahler form  $\check\omega$, we would
like to understand the limiting metric behaviour.

To do so, we use hyperk\"ahler rotation, and to do this we need to
normalize the holomorphic symplectic form, defining
\[
\check\Omega_s^{\nor}=s^{-1}\sqrt{\frac{q_M(\check\omega)}{q_M(\omega)}}
\check\Omega_s.
\]
Then we have
$q_M(\Re\check\Omega^{\nor}_s)=q_M(\Im\check\Omega^{\nor}_s)=
q_M(\check\omega)$. So $\Re\check\Omega^{\nor}_s$,
$\Im\check\Omega^{\nor}_s$ and $\check\omega$ form a hyperk\"ahler
triple, and hence we can hyperk\"ahler rotate to obtain a
hyperk\"ahler manifold with holomorphic two-form
\[
\check\Omega_{s,J}:=\Im \check\Omega^{\nor}_s+\mn\check\omega
=\sqrt{\frac{q_M(\check\omega)}{q_M(\omega)}}\left(\omega-\frac{q_M(\omega,\sigma)}
{q_M(E,\sigma)}E\right)+\mn\check\omega
\]
and K\"ahler form
\[
\check\omega_{s,J}=\Re\check\Omega^{\nor}_s
=\sqrt{\frac{q_M(\check\omega)}{q_M(\omega)}}
\left[\frac{1}{s}\left(\frac{1}{q_M(E,\sigma)}\sigma-\frac{1}{2}\frac{q_M(\sigma)}{q_M(E,\sigma)^2}E\right)+
\frac{s}{2}q_M(\omega)E\right].
\]
We note that the period $\check\Omega_{s,J}$ is in fact independent
of $s$, so we can fix the complex structure on $M$ independent of
$s$. Assume  that $E$ is the first Chern class of a nef line bundle
on $M$ with respect to a complex structure with period
$\check\Omega_{s,J}$, and $\check\omega_{s,J}$ is a K\"ahler class
with respect to this complex structure if $s \geq s_{0}$, for some
$s_{0}\gg 0$. We now take $s=s_{0}\sqrt{\frac{t+1}{t}}$, so that as
$t$ goes to zero, $s$ goes to infinity and we define the rescaled metrics
\[
\check\omega_{t,J}^{\mathrm{nor}}=\sqrt{t(t+1)}\check\omega_{s(t),J}=t\check\omega_{s_{0},J}+
\frac{s_{0}}{2}\sqrt{q_M(\check\omega)q_M(\omega)}E.
\]
So as
$t\rightarrow 0$, $\check\omega_{t,J}^{\mathrm{nor}}$ moves on a straight line
towards $\frac{s_{0}}{2}\sqrt{q_M(\check\omega)q_M(\omega)}E$,
and $\check\omega_{s_{0},J}$ is K\"ahler.

To relate this to the results of this paper, we have the following
conjecture, stated in \cite{HT, verb3}:

\begin{conjecture}\label{conj2.3}
Let $M$ be an irreducible hyperk\"ahler manifold and $\shL$ a
non-trivial nef bundle on $M$, with $q_M(c_1(\shL))=0$. Then $\shL$
induces a holomorphic map $f':M\rightarrow N'$ to a projective variety
$N'$ with $\shL^m \cong f'^*(\O(1))$ for some $m>0$.
\end{conjecture}

 If such a map exists, it is necessarily a holomorphic
Lagrangian fibration. If furthermore $M$ is projective then
$N'=\mathbb{CP}^n$ by \cite{Hw}. This conjecture follows from the log abundance
conjecture if some multiple of $\mathcal{L}$ is effective, and has been studied for example in \cite{AC, COP, HT, verb3}.

Let us suppose this conjecture holds. By choosing $s_{0}$ properly,
we assume that $\frac{s_{0}}{2}\sqrt{q_M(\check\omega)q_M(\omega)}$
is a  integer, and thus
$\frac{s_{0}}{2}\sqrt{q_M(\check\omega)q_M(\omega)}E=f'^{*}\alpha$
for an ample class $\alpha$ on $N'$, where $f'$ and $N'$ are obtained
by Conjecture \ref{conj2.3}. Because of the hyperk\"ahler rotation,
the Riemannian metrics defined by $(\check\Omega^{\nor}_{s},
\check\omega)$ and by $(\check\Omega_{s,J},\check{\omega}_{s,J})$
are the same. Therefore, to understand the Gromov-Hausdorff limit of
the large complex structure limit $(M,\check\Omega^{\nor}_{s}, \check\omega)$
(this is the same that appears in the statement of Theorem \ref{main3}) we can instead consider
$(M,\check\Omega_{s,J},\check{\omega}_{s,J})$. Now
$\check\Omega_{s,J}$ is independent of $s$, so we are simply
changing the K\"ahler class, and the rescaled metrics
$\check\omega_{t,J}=\sqrt{t(t+1)}\check\omega_{s(t),J}=t\check\omega_{s_{0},J}+
f'^{*}\alpha$ move towards $f'^{*}\alpha$ along a straight line.
Therefore we are exactly in the setting of Theorem \ref{main1} and Theorem
\ref{main2}, which describe the Gromov-Hausdorff limit of $(M,\check\omega_{t,J})$
as $t$ goes to zero.
But as remarked in the Introduction, we also have that the diameter of $\check\omega_{t,J}$
is bounded uniformly away from zero and infinity, so if we further rescale the metrics $\check\omega_{t,J}$ to have diameter $1$, then up to a subsequence
the Gromov-Hausdorff limit only changes by a rescaling, and Theorem \ref{main3} follows.

\section{Semi-flat metrics}\label{sect2}

 In this section we discuss semi-flat forms and metrics,
extending some results in \cite{GW, H} to our setting.

In general a closed real $(1,1)$-form $\omega_{SF}$ on an open set
$U\subset M\backslash S$ will be called semi-flat if its restriction
to each torus fiber $M_y\cap U$ with $y\in f(U)$ is a flat metric,
which we will always assume to be cohomologous to $\omega_M|_{M_y}$.
If $\omega_{SF}$ is also K\"ahler then we will call it a semi-flat
metric. Semi-flat forms can also be defined when the fibers $M_y$
are not tori but general Calabi-Yau manifolds, by requiring that the
restriction to each fiber be Ricci--flat (see \cite{ST, To1}). They
were first introduced by Greene-Shapere-Vafa-Yau in \cite{GSVY}.

Fix now a small ball $B\subset N\backslash f(S)$ with coordinates
$y=(y_1,\dots y_m)$,
and consider the preimage $f:U=f^{-1}(B)\to B$. This is a holomorphic family of complex tori, and if $B$ is small enough it has a holomorphic section $\sigma_0$,
which we also fix. We can then define a complex Lie group structure on each fiber $M_y=f^{-1}(y)$ with unit $\sigma_0(y)$.
We claim that this family is locally isomorphic to
a family of the form $f':(B\times\mathbb{C}^{n-m})/\Lambda\to B,$ where $h:\Lambda\to B$ is a lattice bundle with fiber $h^{-1}(y)=\Lambda_y\cong \mathbb{Z}^{2n-2m}$, so that $M_y\cong \mathbb{C}^{n-m}/\Lambda_y.$ To see this, note that each fiber $M_y=f^{-1}(y)$ is a torus biholomorphic to $\mathbb{C}^{n-m}/\Lambda_y$ for some lattice $\Lambda_y$ that varies holomorphically in $y$. We choose a basis $v_1(y),\dots,v_{2n-2m}(y)$ of
this lattice, which varies holomorphically in $y$. Given these lattices we can construct the
family $f'$ by taking the quotient of $B\times\mathbb{C}^{n-m}$
by the $\mathbb{Z}^{2n-2m}$-action given by $(n_1,\dots,n_{2n-2m})\cdot(y,z)=(y,z+\sum_i n_i v_i(y))$, where
$z=(z_1,\dots,z_{n-m})\in\mathbb{C}^{n-m}$. Note that different choices of generators give isomorphic quotients. By construction
the fiber $f'^{-1}(y)$ is biholomorphic to $f^{-1}(y)$ for all $y\in B$. A theorem of Kodaira-Spencer \cite{KSp} (see also \cite[Satz 3.6]{We}) then implies that the families $f$ and $f'$ are locally isomorphic, so up to shrinking $B$ there exists a biholomorphism
$(B\times\mathbb{C}^{n-m})/\Lambda\to U$ compatible with the projections to $B$, proving our claim. With this identification, the section $\sigma_0:B\to U$ is induced by the map $B\to B\times\mathbb{C}^{n-m}$ given by $y\mapsto (y,0)$.

Composing this biholomorphism with the quotient map $B\times\mathbb{C}^{n-m}\to(B\times\mathbb{C}^{n-m})/\Lambda$ by the
$\mathbb{Z}^{2n-2m}$-action we get a holomorphic map $p:B\times\mathbb{C}^{n-m}\to U$ such that $f\circ p(y,z)=y$ for all $(y,z)$, and $p$ is a local isomorphism (the map $p$ is also the universal covering map of $U$).

We now assume that $M$ is projective and $[\omega_M]$ is an integral class,
so each complex torus fiber $M_y$, $y\in B$, can
be polarized by $[\omega_M]$, which gives an ample polarization of type $(d_1,\ldots,d_{n-m})$
for some sequence of integers $d_1|d_2|\cdots|d_{n-m}$.
By \cite{LB}, Proposition 8.1.1,
one can then assume that $\Lambda$ is generated by
$d_1e_1,\ldots,d_{n-m}e_{n-m}, Z_1,\ldots,Z_{n-m}\in \mathbb{C}^{n-m}$, where
$e_1,\ldots,e_{n-m}$ is the standard basis for $\mathbb{C}^{n-m}$.
Furthermore, the matrix $Z$ with columns $Z_1,\ldots,Z_{n-m}$ must
satisfy $Z=Z^t$ and ${\rm Im} Z$ positive definite. Also, on the fibre
$M_y$, the K\"ahler form $\sum_{i,j}\sqrt{-1}(\Im Z)_{ij}dz^i\wedge d\bar z^j$
is cohomologous to $\omega_M|_{M_y}$.
Let
\[
g_{ij}=({\rm Im} Z)^{-1}_{ij}.
\]

Note that $Z$ depends on $y\in B$, as does $g_{ij}$. Recall that
we have the fiber coordinates $z_1,\dots,z_{n-m}$. Consider the function
\[
\eta(y,z)=\sum_{i,j} -{\frac{g_{ij}(y)}{2}}\left((z_i-\bar z_i)(z_j-\bar
z_j) \right).
\]

We would first like to show that $\mn\partial\bar\partial \eta$ is
invariant under translation by flat sections of the Gauss-Manin
connection on $B\times \mathbb{C}^{n-m}$ (this is the connection on this
bundle such that sections of $\Lambda$ are flat sections of the
bundle). It is enough to check invariance under translation by
$\lambda s$ for $s$ one of the generators of $\Lambda$, $\lambda\in
\mathbb{R}$. First, consider the composition of $\eta$ with a
general translation $z_i\mapsto z_i+\tau_i(y)$:
\begin{align*}
&\sum_{i,j} -{\frac{g_{ij}}{2}}\left(
(z_i+\tau_i-\bar z_i-\bar\tau_i)(z_j+\tau_j-\bar z_j-\bar\tau_j)\right)\\
= {} &\eta-\sum_{i,j} {\frac{g_{ij}}{2}}
\left((\tau_i-\bar\tau_i)(z_j-\bar z_j)+(\tau_j-\bar\tau_j)(z_i-\bar
z_i)
+(\tau_i-\bar\tau_i)(\tau_j-\bar\tau_j)\right)\\
={} & \eta-\sum_{i,j} g_{ij}\left((\tau_i-\bar\tau_i)(z_j-\bar
z_j)+ {\frac{1}{2}}(\tau_i-\bar\tau_i)(\tau_j-\bar\tau_j)\right),
\end{align*}
the last equality by the symmetry $g_{ij}=g_{ji}$. We now consider
two cases. If $\tau_i=\lambda\delta_{ik}$ for some $k$, so that
$\tau_i$ is real, then in fact the above formula reduces to
$\eta$, so $\eta$ is itself invariant under this
translation. Secondly, if we take $\tau_i=\lambda Z_{ik}$ for some
$k$, $\lambda\in \mathbb{R}$, we obtain
\begin{align*}
&\eta-\sum_{i,j} ({\rm Im} Z)^{-1}_{ij}\left(
2\lambda \sqrt{-1}({\rm Im}Z)_{ik}(z_j-\bar z_j)-2\lambda^2({\rm Im}Z)_{ik}({\rm Im}Z)_{jk}\right)\\
={} & \eta-\sum_j 2\delta_{jk}\lambda\sqrt{-1}(z_j-\bar z_j)
-2\lambda^2\delta_{jk}({\rm Im}Z)_{jk}.
\end{align*}
Applying $\partial\bar\partial$ kills the correction term, so
$\mn\partial\bar\partial\eta$ is invariant under this action. This means that $\ddbar\eta$ is the pullback under $p$ of a
two-form $\omega_{SF}$ on $U$
\begin{equation}\label{pullsemi}
p^*\omega_{SF}=\ddbar \eta,
\end{equation}
and $\omega_{SF}$ is semi-flat since its restriction to a fiber is $\mn\sum_{i,j}
g_{ij}(y)dz^i\wedge d\bar z^j$, a flat metric on $M_y$ cohomologous to $\omega_M|_{M_y}$. Note that the function $\eta$ on $B\times\mathbb{C}^{n-m}$ has the scaling property
\begin{equation}\label{need}
\eta(y,\lambda z)=\lambda^2\eta(y,z),
\end{equation}
for all $\lambda\in\mathbb{R}$.

We now claim that on $U$ the semi-flat form $\omega_{SF}$ is positive semidefinite.
To check this, it
is enough to check at one point on each fiber, because of the
invariance of this form. We check at the point $z_1=\cdots=z_{n-m}=0$,
where the form is $\mn\sum_{i,j}
g_{ij}dz^i\wedge d\bar z^j$, which is clearly positive semidefinite. It follows that $\omega_{SF}\geq 0$, and moreover that given any K\"ahler metric $\omega'$ on $B$ the form $\omega_{SF}+f^*\omega'$ is a semi-flat K\"ahler metric on $U$.

Suppose now that we have a holomorphic section $\sigma:B\to U$ of the
map $f$. We will denote by $T_\sigma: U\to U$ the fiberwise translation by $\sigma$ (with respect to the section $\sigma_0$). If we choose any local lift of $\sigma$ to $B\times\mathbb{C}^{n-m}$, given by $y\mapsto (y,\ti{\sigma}(y))$, then the translation $T_\sigma$ is induced by the map
$B\times\mathbb{C}^{n-m}\to B\times\mathbb{C}^{n-m}$ given by
$(y,z)\mapsto (y,z+\ti{\sigma}(y))$ (the choice of lift $\ti{\sigma}$ is irrelevant). We also have a map $T_{-\sigma}:U\to U$ given by fiberwise translation by $-\sigma$ (with respect to $\sigma_0$), which is induced by
$(y,z)\mapsto (y,z-\ti{\sigma}(y))$. The two translations are biholomorphisms of $U$ and are inverses to each other.
For later purposes, we will need the following version of the $\de\db$-Lemma, which is analogous to \cite[Lemma 4.3]{GW} (see also \cite[Proposition 3.7]{H}),
except that we work away from the singular fibers.

\begin{proposition}\label{ddbl}
Let $\omega$ be any K\"ahler metric on $U$ cohomologous to $\omega_{SF}$ in $H^2(U,\mathbb{R})$. Then there exist
a holomorphic section $\sigma:B\to U$ of $f$ and a smooth real function $\xi$ on $U$ such that
\begin{equation}\label{ddb}
T_\sigma^*\omega_{SF}-\omega=\ddbar\xi
\end{equation}
on $U$.

 If in addition $\omega$ is also semi-flat, then $\xi$ is
constant on each fiber $M_y$ and is therefore the pullback of a
function from $B$.
\end{proposition}
\begin{proof}
By assumption there is a 1-form $\zeta$ on $U$ such that $$\omega_{SF}-\omega=d\zeta=\partial \zeta^{0,1}
+\overline{\partial} \zeta^{1,0}, \ \ \ \overline{\partial}
\zeta^{0,1}=0,$$ where $\zeta= \zeta^{0,1} + \zeta^{1,0} $ and
$\zeta^{0,1}= \overline{\zeta^{1,0}} $.

 We claim that $(0,1)$-forms
\begin{equation}\label{ddb+1}
\theta_{j}=\mn\ \overline{\partial}\left(\sum_{i=1}^{n-m}g_{ij}(y)(z_{i}-\bar{z}_{i})\right), \ \ \ \ \ \ j=1,\cdots,n-m,
\end{equation}
 are invariant under
translations by flat sections of the Gauss-Manin connection on $B\times\mathbb{C}^{n-m}$, and thus descend to $(0,1)$-forms on $U$.
It is enough to check invariance under translation by $\lambda s$ where $s$
is a generator of $\Lambda$ and $\lambda\in \mathbb{R}$.
First, consider a general translation $z_{i}\mapsto z_{i}+ \tau_{i}(y)$. If
$\tau_i=\lambda\delta_{ik}$ for some $k$, so that $\tau_i$ is real,
then $\theta_{j}$ are invariant. If $\tau_i=\lambda Z_{ik}$ for some
$k$, $\lambda\in \mathbb{R}$, we obtain
\[\begin{split}
\sum_{i=1}^{n-m}g_{ij}(z_{i}&+\lambda Z_{ik}-\bar{z}_{i}-\lambda \bar{Z}_{ik})
 = \sum_{i=1}^{n-m}g_{ij}(z_{i}-\bar{z}_{i})\\
&+ 2\mn\sum_{i=1}^{n-m}
({\rm Im} Z)^{-1}_{ij}
\lambda ({\rm Im}Z)_{ik}=  \sum_{i=1}^{n-m}g_{ij}(z_{i}-\bar{z}_{i})+
 2\lambda \sqrt{-1}\delta_{jk}.
\end{split}\]
Applying $\bar\partial$ kills the correction term, so
 $\theta_{j}$ are invariant, and therefore they define $(0,1)$-forms on
$U$. Since, for any $y\in B$,
\begin{equation}\label{dddd}
p^*\left(\theta_{j}|_{M_y}\right)=-\mn\sum_{i=1}^{n-m}g_{ij}(y)d\bar{z}_{i},
\end{equation}
is fiberwise constant and $g_{ij}$ is non-degenerate, we have that
$[\theta_{i}|_{M_y}]$, $i=1, \cdots, n-m$ is a basis of
$H^{0,1}(M_y)$.

We claim that  there are holomorphic functions $ \sigma_{i}:B\rightarrow
\mathbb{C}$ such that
\begin{equation}\label{ddb+2}
\zeta^{0,1}=\sum_{i=1}^{n-m}\sigma_{i}\theta_{i}+\overline{\partial}h,
\end{equation}
for a complex-valued function $h$ on $U$.
To prove this, note that
$H^{0,1}(U)=H^{1}(U,\mathcal{O}_{U})$ which by the Leray spectral sequence for $f$ is isomorphic to $H^{0}(B,R^{1}f_{*}\mathcal{O}_{U})$
since $H^{k} (B, f_{*}\mathcal{O}_{U})=H^{k} (B,
\mathcal{O}_{B})=0$ for $k\geq 1$. It follows that a $\db$-closed $(0,1)$-form on $U$ represents the zero class if and only if its restriction to $M_y$ represents the zero class in $H^{0,1}(M_y)$ for all $y\in B$.
Consider now the $(0,1)$-forms $d\ov{y}^i$, $1\leq i\leq m$, on $B$ and denote their pullbacks to $U$ by the same symbol. Then at each point of $U$ the forms
$\{\theta_j\}, 1\leq j\leq n-m$ together with $\{d\ov{y}^i\}, 1\leq i\leq m$, form a basis of $(0,1)$-forms. We can then write
$$\zeta^{0,1}=\sum_{j=1}^{n-m} w_j \theta_j + \sum_{i=1}^m h_i d\ov{y}^i,$$
where $w_j, h_i$ are smooth complex functions on $U$. If we now
restrict to a fiber $M_y$ we get $\zeta^{0,1}|_{M_y} =
\sum_{j=1}^{n-m} w_j \theta_j|_{M_y},$ and the functions $w_j$
restricted to $M_y$ can be thought of as functions on
$\mathbb{C}^{n-m}$ which are periodic with period $\Lambda_y$. There
is a holomorphic $T^{2n-2m}$-action on $U$ which is induced by the
action of $\mathbb{R}^{2n-2m}$ on $B\times\mathbb{C}^{n-m}$ given by
$x\cdot (y,z)=(y,z+\sum_j x_j\tau_j(y))$, where $\tau_j(y)$ is a
basis for the lattice $\Lambda_y$ (the choice of which is
irrelevant). If $\alpha$ is a function or differential form on $U$
or $M_y$, we will denote by $\ti{\alpha}$ its average with respect
to the $T^{2n-2m}$-action. In particular, if $\alpha$ is a function
on $U$ then $\ti{\alpha}$ is the pullback of a function from $B$. We
now call $\sigma_j=\ti{w}_j$, $1\leq j\leq n-m$, which are functions
of $y\in B$ only. We clearly have that $\ti{\theta}_j=\theta_j$ and
$\widetilde{d\ov{y}^i}=d\ov{y}^i$, so
$$\widetilde{\zeta^{0,1}|_{M_y}}=\sum_{j=1}^{n-m} \sigma_j(y) \theta_j|_{M_y}.$$
Now the $T^{2n-2m}$-action on $M_y$ is generated by holomorphic vector fields and therefore acts trivially on the Dolbeault cohomology $H^{0,1}(M_y)$, which implies that
$$\left[\zeta^{0,1}|_{M_y}\right]=\left[\widetilde{\zeta^{0,1}|_{M_y}}\right]=
\sum_{j=1}^{n-m}\sigma_j(y)\left[\theta_j|_{M_y}\right],$$
in $H^{0,1}(M_y)$ for all $y\in B$. If we show that the $\sigma_j(y)$ are holomorphic, then the $(0,1)$-form $\zeta^{0,1}-\sum_j \sigma_j(y)\theta_j$ on $U$ would be $\db$-closed and cohomologous to zero in $H^{0,1}(U)$, thus proving \eqref{ddb+2}.

Call now $V_j$, $1\leq j\leq n-m$ and $W_i$, $1\leq i\leq m$ the
$T^{2n-2m}$-invariant $(0,1)$-type vector fields on $U$ which are the dual basis to $\theta_j, d\ov{y}^i$.
We have that $V_j=\mn \sum_{k=1}^{n-m}g^{jk} \frac{\de}{\de \ov{z}_k}$, where $g^{jk}$ is the inverse matrix of $g_{jk}$, and the vector fields $\frac{\de}{\de \ov{z}_k}$ are well-defined on $U$. We will not need the explicit formula for $W_i$, but just the fact that if a function $f$ on $U$ is the pullback of a function on $B$ then $W_i(f)=\frac{\de f}{\de \ov{y}_i}$.

To see why $\sigma_j(y)$ is holomorphic, compute
\[\begin{split}
0=\db\zeta^{0,1}=&\sum_{i,j}W_i(w_j) d\ov{y}^i\wedge\theta_j+\sum_{i,j} V_i(w_j)
\theta_i\wedge\theta_j \\
&+\sum_{i,j}W_j(h_i)
d\ov{y}^j\wedge d\ov{y}^i+\sum_{i,j} V_j(h_i) \theta_j\wedge
d\ov{y}^i.
\end{split}\]
Since each $V_j$ is a linear combination of $\frac{\de}{\de \ov{z}_k}$,
we have that the functions $V_i(w_j)$ and $V_j(h_i)$ have
average zero on each fiber. Taking
the average then gives
$$0=\db \widetilde{\zeta^{0,1}}=\sum_{i,j} \frac{\de \sigma_j}{\de \ov{y}_i}
d\ov{y}^i\wedge\theta_j+\sum_{i,j}\frac{\de \tilde{h}_i}{\de \ov{y}_j} d\ov{y}^j\wedge d\ov{y}^i.$$
Since the forms $d\ov{y}^i\wedge\theta_j$ and $d\ov{y}^j\wedge d\ov{y}^i$ are linearly independent at every point, this implies that $\sigma_j(y)$ are indeed holomorphic.

Let now $T_{\sigma}: U\to U$ be the translation
induced by the section  $\sigma=(p\circ\sigma_{1}, \cdots,
p\circ\sigma_{n-m})$, where $p: B\times\mathbb{C}^{n-m}\rightarrow
U$ is the quotient map. Since
\[\begin{split} \sum_{i,j} &-{\frac{g_{ij}}{2}}\left(
(z_i+\sigma_i-\bar z_i-\bar\sigma_i)(z_j+\sigma_j-\bar z_j-\bar\sigma_j)\right)\\
= {} &\eta-\sum_{i,j} {\frac{g_{ij}}{2}}
\left((\sigma_i-\bar\sigma_i)(z_j-\bar
z_j)+(\sigma_j-\bar\sigma_j)(z_i-\bar z_i)
+(\sigma_i-\bar\sigma_i)(\sigma_j-\bar\sigma_j)\right)\\
={} & \eta-\sum_{i,j} g_{ij}\left((\sigma_i-\bar\sigma_i)(z_j-\bar
z_j)+
{\frac{1}{2}}(\sigma_i-\bar\sigma_i)(\sigma_j-\bar\sigma_j)\right),
\end{split}\] we have
\[\begin{split}
p^*T_{\sigma}^{*}\omega_{SF}-p^*\omega_{SF}& =  -\sqrt{-1}\partial\overline{\partial}\sum_{i,j} g_{ij}(\sigma_i-\bar\sigma_i)(z_j-\bar
z_j)+\sqrt{-1}\partial\overline{\partial}\Phi(y) \\
 &= p^*\left(-\partial\sum_{i}\sigma_{i}\theta_{i}-
\overline{\partial}\sum_{i}\ov{\sigma_{i}\theta_{i}}\right)
+\sqrt{-1}\partial\overline{\partial}\Phi(y),\end{split}\]
where $\Phi(y)=-\sum_{i,j}\frac{g_{ij}}{2}(\sigma_i-\bar\sigma_i)(\sigma_j-\bar\sigma_j)$ is a real function of $y$ only.
We have just proved that
$$\omega_{SF}-\omega=\partial \zeta^{0,1}
+\overline{\partial}\ \ov{\zeta^{0,1}}=\de\sum_i \sigma_i \theta_i+\de\db h
+\db \sum_i \ov{\sigma_i \theta_i} +\db\de \ov{h}.$$
Thus
$$p^*T_{\sigma}^{*}\omega_{SF}-p^*\omega=p^*\sqrt{-1}\partial\overline{\partial}(2\mathrm{Im}h+\Phi), $$
which proves \eqref{ddb} with $\xi=2\mathrm{Im}h+\Phi$.
\end{proof}

\section{Estimates and smooth convergence}\label{sect3} 

In this section we prove a priori estimates of all orders for the
Ricci--flat metrics $\ti{\omega}_t$ which are uniform on compact
sets of $M\backslash S$, and then use these to prove Theorem
\ref{main1}. These estimates improve the results in \cite{To1}, and
use crucially the assumptions that $M$ is projective and that the
smooth fibers $M_y$ are tori. 

From now on we fix a small ball $B\subset N\backslash f(S)$, and as before
we call $U=f^{-1}(B)$ and we have the holomorphic covering map $p:B\times\mathbb{C}^{n-m}\to U$, with $f\circ p(y,z)=y$ where $(y,z)=(y_1,\dots,y_m,z_1,\dots,z_{n-m})$
the standard coordinates on $B\times\mathbb{C}^{n-m}$.

\begin{lemma}\label{c2es} There is a constant $C$ such that on $U$ the Ricci--flat metrics $\ti{\omega}_t$ satisfy
\begin{equation}\label{c2}
C^{-1}(\omega_0+t\omega_M)\leq \ti{\omega}_t\leq C(\omega_0+t\omega_M),
\end{equation}
for all small $t>0$.
\end{lemma}
\begin{proof}
This estimate is contained in the second-named author's work \cite{To1}, although it is not explicitly stated there.
To see this, start from \cite[(3.24)]{To1}, which gives a constant $C$ so that on $U$ we have
$$C^{-1}(t\omega_M)\leq \ti{\omega}_t.$$
Then use \cite[Lemma 3.1]{To1} to get
$$C^{-1}\omega_0\leq \ti{\omega}_t,$$
and so adding these two inequalities we get
$$C^{-1}(\omega_0+t\omega_M)\leq \ti{\omega}_t,$$
or in other words $\tr{\ti{\omega}_t}{\omega_t}\leq C$ on $U$, where
$\omega_t=\omega_0+t\omega_M$ as before.
To get the reverse inequality, we note that on $U$ we have
$$\tr{\omega_t}{\ti{\omega}_t}\leq (\tr{\ti{\omega}_t}{\omega_t})^{n-1} \frac{\ti{\omega}_t^n}{\omega_t^n}
\leq C\frac{\ti{\omega}_t^n}{\omega_t^n}\leq C,$$
where the last inequality follows from \cite[(3.23)]{To1}. We thus get the reverse inequality
$$\ti{\omega}_t\leq C(\omega_0+t\omega_M),$$
thus proving \eqref{c2}.
\end{proof}

We now let
$\lambda_t:B\times\mathbb{C}^{n-m}\to B\times\mathbb{C}^{n-m}$ be the dilation
$$\lambda_t(y,z)=\left(y,\frac{z}{\sqrt{t}}\right),$$
which takes the lattice $\sqrt{t}\Lambda_y$ to $\Lambda_y$.
If we pull back the K\"ahler potential $\vp_t$ on $U$ via $p$ we get a function
$\vp_t\circ p$ on $B\times\mathbb{C}^{n-m}$ which is periodic in
$z$ with period $\Lambda_y$, i.e. $\vp_t\circ p (y,z+\ell)=\vp_t\circ p (y,z)$
for all $\ell\in\Lambda_y$. The function $\vp_t\circ p\circ \lambda_t$
is then periodic in $z$ with period $\sqrt{t}\Lambda_y$.
Note that since $\omega_0$ is the pullback of a metric from $N\backslash f(S)$, we have $\lambda_t^*p^*\omega_0=p^*\omega_0.$

Recall now that we have a positive semidefinite semi-flat form $\omega_{SF}$ on $U$, and that $\omega_0+\omega_{SF}$ is then a semi-flat K\"ahler metric on $U$. Since $U$ is diffeomorphic to a product $B\times M_y$, it follows that $\omega_{SF}$ and $\omega_M$ are cohomologous on $U$. We now apply Proposition \ref{ddbl} and get a holomorphic section $\sigma:B\to U$ and a real function $\xi$ on $U$ such that
\begin{equation}\label{dd}
T_\sigma^*\omega_{SF}-\omega_M=\ddbar\xi
\end{equation}
on $U$, where $T_\sigma$ is the fiberwise translation by $\sigma$.

\begin{lemma}
There is a constant $C$ such that on the whole of $B\times\mathbb{C}^{n-m}$ we have
\begin{equation}\label{c2b}
C^{-1}p^*(\omega_0+\omega_{SF})\leq \lambda_t^*p^*T_{-\sigma}^*\ti{\omega}_t\leq Cp^*(\omega_0+\omega_{SF}),
\end{equation}
for all small $t>0$.
\end{lemma}

\begin{proof}
First of all notice that after replacing $U$
with a slightly smaller open
set, the semi-flat metric $\omega_0+\omega_{SF}$ is uniformly equivalent to $\omega_M$, which implies that
\begin{equation}\label{uno}
C^{-1}(\omega_0+t\omega_{SF})\leq \omega_0+t\omega_M\leq C(\omega_0+t\omega_{SF}),
\end{equation}
for all small $t>0$.
Thanks to Lemma \ref{c2es} on $U$ we have that
$$C^{-1}(\omega_0+tT_{-\sigma}^*\omega_M)\leq T_{-\sigma}^*\ti{\omega}_t\leq C(\omega_0+tT_{-\sigma}^*\omega_M),$$
and since $T_{-\sigma}^*\omega_M$ is uniformly equivalent to $\omega_M$ we also have that
\begin{equation}\label{due}
C^{-1}(\omega_0+t\omega_M)\leq T_{-\sigma}^*\ti{\omega}_t\leq C(\omega_0+t\omega_M),
\end{equation}
and combining \eqref{uno} and \eqref{due} we get
\begin{equation}\label{c2m}
C^{-1}(\omega_0+t\omega_{SF})\leq T_{-\sigma}^*\ti{\omega}_t\leq C(\omega_0+t\omega_{SF}),
\end{equation}
on $U$.
If we pull back \eqref{c2m} by $p\circ \lambda_t$ we get
\begin{equation}\label{equiv2}
C^{-1}(p^*\omega_0+t\lambda_t^*p^*\omega_{SF})\leq
\lambda_t^*p^*T_{-\sigma}^*\ti{\omega}_t\leq C(p^*\omega_0+t\lambda_t^*p^*\omega_{SF}),
\end{equation}
on all of $B\times\mathbb{C}^{n-m}$.
We claim that on the whole of $B\times\mathbb{C}^{n-m}$ we have that
\begin{equation}\label{need1}
t\lambda_t^*p^*\omega_{SF}=
p^*\omega_{SF}.
\end{equation}
In fact, the construction of $\omega_{SF}$ in section \ref{sect2} gives that
$p^*\omega_{SF}=\ddbar\eta,$
for a function $\eta$ on $B\times\mathbb{C}^{n-m}$
that satisfies
\begin{equation}\label{need2}
\eta\circ\lambda_t (y,z)=\eta\left(y,\frac{z}{\sqrt{t}}\right)=\frac{1}{t}\eta(y,z),
\end{equation}
for all $(y,z)$ in $B\times\mathbb{C}^{n-m}$ and any $t>0$.
It follows then that
\begin{equation}\label{need3}
t\lambda_t^*p^*\omega_{SF}=t\lambda_t^*\ddbar\eta=
t\ddbar(\eta\circ\lambda_t)=\ddbar\eta=p^*\omega_{SF},
\end{equation}
as claimed.
Combining \eqref{equiv2} and \eqref{need1} we get the bound \eqref{c2b}.
\end{proof}

\begin{proposition}\label{estimates}
Given any compact set $K$ in $B\times\mathbb{C}^{n-m}$
and any $k\geq 0$ there exists a constant $C$
independent of $t>0$ such that
\begin{equation}\label{estim}
\|\lambda_t^*p^*T_{-\sigma}^*\ti{\omega}_t \|_{C^{k}(K,\delta)}\leq C,
\end{equation}
where $\delta$ is the Euclidean metric
on $B\times\mathbb{C}^{n-m}$.
\end{proposition}

\begin{proof}
We pull back \eqref{ma} via $T_{-\sigma}\circ p\circ\lambda_t$ and get
\[\begin{split}
(\lambda_t^*p^*T_{-\sigma}^*\ti{\omega}_t)^n(y,z)&=c_t t^{n-m}(\lambda_t^*p^* T_{-\sigma}^*\omega_M)^n(y,z)\\
&=c_t (p^*T_{-\sigma}^*\omega_M)^n\left(y,\frac{z}{\sqrt{t}}\right),
\end{split}\]
since the pullback under $\lambda_t$ of any volume form $f(y,z)dy^1\wedge\dots\wedge d\ov{z}^{n-m}$ on $B\times\mathbb{C}^{n-m}$
equals
$t^{m-n} f(y,\frac{z}{\sqrt{t}}) dy^1\wedge\dots\wedge d\ov{z}^{n-m}.$
We now claim that in fact we have
$$(p^*T_{-\sigma}^*\omega_M)^n\left(y,\frac{z}{\sqrt{t}}\right)=(p^*T_{-\sigma}^*\omega_M)^n(y,z).$$
To see this, consider the $(n,0)$-form
$$dy^1\wedge\dots\wedge dy^m\wedge dz^1\wedge\dots\wedge dz^{n-m}$$
on $B\times\mathbb{C}^{n-m}$. This form is invariant under the $\mathbb{Z}^{2n-2m}$-action described above
$$(n_1,\dots,n_{2n-2m})\cdot(y,z)=(y,z+\sum_i n_i v_i(y)),$$
where $(y,z)=(y_1,\dots,y_m,z_1,\dots,z_{n-m})$, and so it descends to
a holomorphic $(n,0)$-form to the quotient $(B\times \mathbb{C}^{n-m})/\Lambda$
and using the biholomorphism with $U$ we get a holomorphic $(n,0)$-form
$\Omega$ on $U$. We can then consider the volume form
$(\mn)^{n^2}\Omega\wedge\ov{\Omega}$, and we have
$$T_{-\sigma}^*\omega_M^n=h\cdot (\mn)^{n^2}\Omega\wedge\ov{\Omega},$$
where $h$ is a smooth positive function on $U$. Taking $\ddbar\log$
of both sides we get
$$\ddbar\log h=\ddbar\log\frac{T_{-\sigma}^*\omega_M^n}{(\mn)^{n^2}\Omega\wedge\ov{\Omega}}=0,$$
since $T_{-\sigma}^*\omega_M$ is Ricci--flat and $\Omega$ is a
holomorphic $(n,0)$-form. So $\log h$ is pluriharmonic on $U$, and
this implies that its restriction to any fiber $M_y$ with $y\in B$
is constant. Pulling back via $p$ we get
$$(p^*T_{-\sigma}^*\omega_M^n)(y,z)=(h\circ p)(y,z) (\mn)^{n^2} dy^1\wedge\dots\wedge d\ov{z}^{n-m},$$
but since $h$ is constant along the fibers of $f$ and $p$ is compatible with the projection to $B$ we get that
the function $(h\circ p)(y,z)$ on $B\times\mathbb{C}^{n-m}$ is independent of $z$.
In particular we have
$$(p^*T_{-\sigma}^*\omega_M)^n\left(y,\frac{z}{\sqrt{t}}\right)=(p^*T_{-\sigma}^*\omega_M)^n(y,z),$$
and so the rescaled metrics $\lambda_t^*p^*T_{-\sigma}^*\ti{\omega}_t$ satisfy
the nondegenerate complex Monge-Amp\`ere equation
$$(\lambda_t^*p^*T_{-\sigma}^*\ti{\omega}_t)^n=(p^*\omega_0+t\lambda_t^*p^*T_{-\sigma}^*\omega_M + \ddbar \ti{\vp}_t)^n=
c_t (p^*T_{-\sigma}^*\omega_M)^n$$
on $B\times\mathbb{C}^{n-m}$, where we have set
$$\ti{\vp}_t=\vp_t\circ T_{-\sigma}\circ p\circ \lambda_t.$$
 We claim that the estimates \eqref{estim}
hold.
To see this, we use \eqref{dd} and get
\begin{equation}\label{need4}
p^*\omega_{SF}=p^*T_{-\sigma}^*\omega_M + p^* T_{-\sigma}^*\ddbar\xi,
\end{equation}
for a function $\xi$ on $U$.
On $B\times\mathbb{C}^{n-m}$ we can then use \eqref{need3} and \eqref{need4} and write
\begin{equation}\label{ut}
\begin{split}
\lambda_t^*p^*T_{-\sigma}^*\ti{\omega}_t&=p^*\omega_0+t\lambda_t^*p^*T_{-\sigma}^*\omega_M
+\ddbar\ti{\vp}_t\\
&=p^*\omega_0+t\lambda_t^*p^*(\omega_{SF}-T_{-\sigma}^*\ddbar\xi)+\ddbar\ti{\vp}_t\\
&=p^*\omega_0+p^*\omega_{SF}-t\lambda_t^*p^*T_{-\sigma}^*\ddbar\xi+\ddbar\ti{\vp}_t\\
&=p^*(\omega_0+\omega_{SF})+\ddbar u_t,
\end{split}
\end{equation}
where for simplicity we write
$u_t=\ti{\vp}_t-t(\xi\circ T_{-\sigma}\circ p\circ \lambda_t).$
The functions $u_t$ are uniformly bounded in $C^0(B\times\mathbb{C}^{n-m})$ because
of the $L^\infty$ bound for $\vp_t$ from \cite{DP, EGZ2} and because $\xi$ is a fixed
function on $U$.
The functions $u_t$ satisfy the complex Monge-Amp\`ere equations
\begin{equation}\label{maa}
(p^*\omega_0+p^*\omega_{SF}+\ddbar u_t)^n=c_t (p^*T_{-\sigma}^*\omega_M)^n
\end{equation}
on $B\times\mathbb{C}^{n-m}$, and on any compact subset $K$ of $B\times\mathbb{C}^{n-m}$ the K\"ahler metric
$p^*(\omega_0+\omega_{SF})$ is $C^\infty$ equivalent to the Euclidean metric $\delta$ (with constants that depend only on $K$).
The bounds \eqref{c2b} imply that
$$C^{-1}\delta\leq p^*(\omega_0+\omega_{SF})+\ddbar u_t\leq C\delta,$$
on $K$ for all small $t>0$, where $C$ depends on $K$. The constants $c_t$ are bounded uniformly and away from zero,
because by definition we have
\begin{equation}\label{normaliz}
\lim_{t\to 0}c_t=\binom{n}{m}\frac{\int_M \omega_0^m\wedge\omega_M^{n-m}}{\int_M \omega_M^n}>0,
\end{equation}
see also \cite{EGZ2}, \cite[(2.6)]{To1}.
After shrinking $K$ slightly we can then apply the Evans-Krylov theory (as explained for example in \cite{GT, Si}) and
Schauder estimates to get higher order estimates $\|u_t\|_{C^k(K,\delta)}\leq C(k)$ for all $k\geq 0$, thus proving \eqref{estim}.
\end{proof}

\begin{lemma}\label{sectcurv}
Given any compact set $K\subset M\backslash S$ there is a constant $C_K$ such that
the sectional curvature of $\ti{\omega}_t$ satisfies
\begin{equation}\label{sectbound2}
\sup_K |\mathrm{Sec}(\ti{\omega}_t)|\leq C_K,
\end{equation}
for all small $t>0$.
\end{lemma}

\begin{proof}
We can assume that $K$ is sufficiently small so that $f(K)\subset B$ for a ball $B$ as before, and that there is a compact set $K'\subset B\times\mathbb{C}^{n-m}$ so that $p:K'\to T_\sigma(K)$ is a biholomorphism.
We then have
\[\begin{split}
\sup_K |\mathrm{Sec}(\ti{\omega}_t)|&=
\sup_{T_\sigma(K)}|\mathrm{Sec}(T_{-\sigma}^*\ti{\omega}_t)|=\sup_{K'}|\mathrm{Sec}(p^*T_{-\sigma}^*\ti{\omega}_t)|\\
&=\sup_{\lambda_t^{-1}(K')}|\mathrm{Sec}(\lambda_t^*p^*T_{-\sigma}^*\ti{\omega}_t)|.
\end{split}\]
For $t>0$ small enough, the sets $\lambda_t^{-1}(K)$ are all contained in a fixed compact set $K''\subset B\times\mathbb{C}^{n-m}$. From \eqref{c2b} and \eqref{estim} we then get a uniform bound for the sectional curvatures of
$\lambda_t^*p^*T_{-\sigma}^*\ti{\omega}_t$ on $K''$, and this proves \eqref{sectbound2}.
\end{proof}

\begin{lemma}\label{conv2}
Given any compact set $K$ in $B\times\mathbb{C}^{n-m}$
and any $k\geq 0$ there exists a constant $C$
independent of $t>0$ such that
\begin{equation}\label{estim2}
\|p^*T_{-\sigma}^*\ti{\omega}_t \|_{C^{k}(K,\delta)}\leq C,
\end{equation}
where $\delta$ is the Euclidean metric
on $B\times\mathbb{C}^{n-m}$.
\end{lemma}

\begin{proof}
Given $K$, for all $t>0$ small enough the sets
$\lambda_t^{-1}(K)$ are all contained in a fixed compact set $K'\subset B\times\mathbb{C}^{n-m}$. We wish to deduce \eqref{estim2} from \eqref{estim}.
To see this, write on $B\times\mathbb{C}^{n-m}$
\[\begin{split}
\lambda_t^*p^*T_{-\sigma}^*\ti{\omega}_t=&\mn\bigg(\sum_{i,j} A_{i\ov{j}}(t,y,z)dz^i\wedge d\ov{z}^j+\sum_{i,j} B_{i\ov{j}}(t,y,z)dy^i\wedge d\ov{y}^j\\
&+\sum_{i,j} C_{i\ov{j}}(t,y,z)dy^i\wedge d\ov{z}^j
+\sum_{i,j} D_{i\ov{j}}(t,y,z)dz^i\wedge d\ov{y}^j\bigg).
\end{split}\]
Thanks to \eqref{estim}, on $K'$ the coefficents $A,B,C,D$ satisfy uniform
$C^k$ estimates in the variables $(y,z)$ independent of $t$. We then pull back this equation via the map $\lambda_{1/t}$ (the inverse of $\lambda_t$) and get
\[\begin{split}
p^*T_{-\sigma}^*\ti{\omega}_t&=\mn\bigg(t\sum_{i,j} A_{i\ov{j}}(t,y,z\sqrt{t})dz^i\wedge d\ov{z}^j+\sum_{i,j} B_{i\ov{j}}(t,y,z\sqrt{t})dy^i\wedge d\ov{y}^j\\
&+\sqrt{t}\sum_{i,j} C_{i\ov{j}}(t,y,z\sqrt{t})dy^i\wedge d\ov{z}^j
+\sqrt{t}\sum_{i,j} D_{i\ov{j}}(t,y,z\sqrt{t})dz^i\wedge d\ov{y}^j\bigg),
\end{split}\]
and the new coefficients are uniformly bounded in $C^k$ on $K$, thus proving \eqref{estim2}.
\end{proof}

\begin{proposition}\label{conv1}
As $t$ goes to zero we have
$$\ti{\omega}_t\to f^*\omega$$ in $C^\infty_{loc}(M\backslash S,\omega_M)$,
where $\omega=\omega_N+\ddbar\vp$ is a K\"ahler metric on
$N\backslash f(S)$ with $\Ric(\omega)=\omega_{\rm WP}$ as in
Theorem \ref{main1}.
\end{proposition}

\begin{proof}
Recall that $\ti{\omega}_t=\omega_0+t\omega_M+\ddbar\vp_t$,
so that
$$p^*T_{-\sigma}^*\ti{\omega}_t=p^*\omega_0+tp^*T_{-\sigma}^*\omega_M+\ddbar(\vp_t\circ T_{-\sigma}\circ p).$$
We now fix a compact set $K\subset M\backslash S$, which we can assume is sufficiently small so that $f(K)\subset B$ for a ball $B$ as before, and that there is a compact set $K'\subset B\times\mathbb{C}^{n-m}$ such that $p:K'\to T_{\sigma}(K)$ is a biholomorphism.
From \eqref{estim2} (together with the $L^\infty$ bound for $\vp_t$ from \cite{DP, EGZ2}) we see that
$$\|\vp_t\circ T_{-\sigma}\circ p\|_{C^k(K',\delta)}\leq C(k),$$
and therefore also
\begin{equation}\label{cinfty}
\|\vp_t\|_{C^k(K,\omega_M)}\leq C(k),
\end{equation}
since $T_{-\sigma}\circ p:K'\to K$ is a fixed biholomorphism.
From \cite{To1} we know that $\vp_t\to f^*\vp$ in $C^{1,\alpha}_{loc}(M\backslash S, \omega_M)$, and so \eqref{cinfty} implies that $\vp_t\to f^*\vp$ in $C^\infty_{loc}(M\backslash S, \omega_M)$, and therefore that
$\ti{\omega}_t\to f^*\omega$ in $C^\infty_{loc}(M\backslash S,\omega_M)$.
\end{proof}

 As a corollary of this, for any compact subset $K\subset
M\backslash S$,  there is a positive function $\ve(t)$ which goes to
zero as $t\rightarrow 0$, such that
\begin{equation}\label{conseq}
f^{*}\omega -\ve(t)\omega_M
    \leq  \tilde{\omega}_{t} \leq f^{*}\omega +\ve(t)\omega_M
\end{equation}
 on $K$, as well as
\begin{equation}\label{conseq2}
e^{-\ve(t)}f^{*}\omega
    \leq  \tilde{\omega}_{t}.
\end{equation}

We now finish the proof of Theorem \ref{main1}. We have
already proved the first two statements in Proposition \ref{conv1} and Lemma \ref{sectcurv}, and it remains to prove \eqref{rescaled}. First, we need the following lemma. 
\begin{lemma}\label{conv4}
As $t$ goes to zero we have
\begin{equation}\label{conv}
\lambda_t^*p^*T_{-\sigma}^*\ti{\omega}_t\to p^*(\omega_{SF}+f^*\omega)
\end{equation}
in $C^\infty_{loc}(B\times\mathbb{C}^{n-m},\delta)$,
where $\delta$ is the Euclidean metric.
\end{lemma}

\begin{proof}
Recall that from \eqref{ut} we see that on $B\times\mathbb{C}^{n-m}$
$$\lambda_t^*p^*T_{-\sigma}^*\ti{\omega}_t=p^*(\omega_0+\omega_{SF})+\ddbar u_t,$$
where the functions $u_t=\ti{\vp}_t-t\lambda_t^*p^*T_{-\sigma}^*\xi$ have uniform $C^\infty$ bounds on compact sets.
We need to show that as $t$ goes to zero we have $u_t\to(f\circ p)^*\vp$ in $C^\infty_{loc}(B\times\mathbb{C}^{n-m},\delta)$, where $f^*\vp$ is the $C^{1,\alpha}$ limit of $\vp_t$ from \cite{To1}. To prove this we need another estimate from the second-named author's work \cite[(3.9)]{To1}, which implies that there is a constant $C$ (that depends on the initial choice of $B$) so that for all $0<t\leq 1$ we have
\begin{equation}\label{oscest}
\sup_{y\in B} \mathrm{osc}_{M_y}\vp_t\leq C t.
\end{equation}
We now use this together with the fact that $\vp_t\to f^*\vp$ in $C^0$ to get
that for any $(y,z)$ in $B\times\mathbb{C}^{n-m}$ we have
\[\begin{split}|\ti{\vp}_t(y,z)-(f\circ p)^*\vp(y,z)|&=\left|\vp_t\circ T_{-\sigma}\circ p\left(y, \frac{z}{\sqrt{t}}\right)-\vp(y)\right|\\
&\leq \left|\vp_t\circ p\left(y,\frac{z}{\sqrt{t}}-\ti{\sigma}(y)\right)-\vp_t\circ p(y,z)\right|\\
&\ \ \ \ +|\vp_t\circ p(y,z)-((f^*\vp)\circ p)(y)|\\
&\leq Ct+\sup_U |\vp_t-f^*\vp|,
\end{split}\]
where in the last line we used \eqref{oscest} because the points $p(y,\frac{z}{\sqrt{t}}-\ti{\sigma}(y))$ and $p(y,z)$ lie in the same fiber $M_y$. Letting $t$ go to zero we see that $\ti{\vp}_t\to (f\circ p)^*\vp$ in $C^0(B\times\mathbb{C}^{n-m})$.
On the other hand we have that
$t\lambda_t^*p^*\xi\to 0$
in $C^0(B\times\mathbb{C}^{n-m})$, and so $u_t\to (f\circ p)^*\vp$
in $C^0(B\times\mathbb{C}^{n-m})$.
Thanks to the higher order estimates for $u_t$, we also have that
$u_t\to (f\circ p)^*\vp$ in $C^\infty_{loc}(B\times\mathbb{C}^{n-m},\delta)$, up to shrinking $B$ slightly.
\end{proof}

 We can now complete the proof of Theorem \ref{main1}.
\begin{proof}
Recall that thanks to Lemma \ref{conv4}, on $B\times\mathbb{C}^{n-m}$ we can write
$$\lambda_t^*p^*T_{-\sigma}^*\ti{\omega}_t-p^*(\omega_{SF}+f^*\omega)=E_t,$$
where the error term $E_t$ is a $(1,1)$-form that goes to zero smoothly on compact sets.
From \eqref{need1} we also have that
$$E_t=\lambda_t^*p^*(T_{-\sigma}^*\ti{\omega}_t-f^*\omega-t\omega_{SF}).$$
If we restrict the form $T_{-\sigma}^*\ti{\omega}_t-f^*\omega-t\omega_{SF}$ to a fiber $M_y$ and divide by $t$ we get
$$\frac{E_t}{t}\bigg|_{\{y\}\times\mathbb{C}^{n-m}}=\lambda_t^*p^*\left(\frac{T_{-\sigma}^*\ti{\omega}_t|_{M_y}}{t}-\omega_{SF,y}\right)$$
Pulling back this via the map $\lambda_{1/t}$ (the inverse of $\lambda_t$) we get
$$\frac{\lambda_{1/t}^*E_t}{t}\bigg|_{\{y\}\times\mathbb{C}^{n-m}}=p^*\left(\frac{T_{-\sigma}^*\ti{\omega}_t|_{M_y}}{t}-\omega_{SF,y}\right).$$
Explicitly we have $\lambda_{1/t}(y,z)=(y,z\sqrt{t})$, which implies that $\lambda_{1/t}^*dz^i=\sqrt{t} dz^i$, and so
$$\frac{\lambda_{1/t}^*E_t}{t}\bigg|_{\{y\}\times\mathbb{C}^{n-m}}(y,z)=E_t\bigg|_{\{y\}\times\mathbb{C}^{n-m}}(y,z\sqrt{t}),$$
which goes to zero smoothly as $t$ approaches zero, uniformly in $y$. It follows that $\frac{T_{-\sigma}^*\ti{\omega}_t|_{M_y}}{t}$ converges smoothly to $\omega_{SF,y}$, and the convergence is uniform as $y$ varies on compact sets of $N\backslash f(S)$. Pulling back via $T_{\sigma}$, and using the fact that $T_\sigma^*\omega_{SF,y}=\omega_{SF,y}$, we see that also $\frac{\ti{\omega}_t|_{M_y}}{t}$ converges smoothly to $\omega_{SF,y}$, as desired.
\end{proof}

\begin{remark} Note that in particular we get the estimate
$$\sup_{M_y} \left|\nabla (\ti{\omega}|_{M_y})\right|^2_{\omega_M}\leq Ct^2,$$
which improves \cite[(2.11)]{To1}.
\end{remark}

\section{Gromov-Hausdorff convergence}\label{sect4}
 In this section we study the collapsed Gromov-Hausdorff
limits of the Ricci--flat metrics $\ti{\omega}_t$ and prove Theorem
\ref{main2}.

\begin{lemma}\label{p00.2}
There is an open subset $X_{0}\subset X$ such that $(X_{0}, d_{X})$
is locally isometric to $(N_{0}, \omega)$ where $N_{0}=N\backslash
f(S) $, i.e. there is a homeomorphism $\phi:  N_{0} \rightarrow
X_{0}$ such that, for any $y\in N_{0}$, there is a neighborhood
$B_{y}\subset N_{0}$ of $y$ satisfying that, if $y_{1}$ and
$y_{2}\in B_{y}$,
  $$d_{\omega}(y_{1},y_{2})=d_{X}(\phi(y_{1}),\phi(y_{2})).$$
  Furthermore, for any $y\in N_{0}$, there is a compact  neighborhood $B\subset N_{0}$ and
  a
  holomorphic section $s:B\rightarrow f^{-1}(B)$, i.e., $f\circ s= {\rm id}$, such
  that $s(y)\rightarrow \phi(y)$ under the Gromov-Hausdorff
  convergence of $(M, \tilde{\omega}_{t_{k}}) $  to
   $(X, d_{X}) $.
\end{lemma}

\begin{proof}
Let $A$ be a countable dense subset of $N_{0}$, and $K\subset N_{0}$
be a compact subset with the interior $\Int K $ non-empty.  Let
$\{B_{i}\} $ be a finite  covering of $K$ with small Euclidean balls
such that each of the concentric balls $B_i'$ of half radius still cover $K$.
 Let $s_{i}: B_{i} \rightarrow f^{-1}(B_{i})$ be
sections on $B_{i}$, i.e., holomorphic maps with $f\circ s_{i}={\rm
id} $.

  Now, we
  define a map $\phi$ from $A\cap K=\{a_{1}, a_{2}, \cdots\}$ to
  $X$. Suppose that the point $a_1$ lies inside the ball $B_i'$, and
consider the points $s_i(a_1)$ inside $M$. Under the Gromov-Hausdorff convergence
of $(M,\ti{\omega}_{t_k})$ to $(X,d_X)$, a subsequence of these points converges to a point $b_1$ in $X$, because the diameter of $(M,\ti{\omega}_{t_k})$ is uniformly bounded.
If $a_{1}$ also lies inside another ball $B_{j}'$, then
    \eqref{rescaled} (or also \cite[(2.10)]{To1}) shows that
    $d_{\tilde{\omega}_{t_{k}}}(s_{i}(a_{1}),s_{j}(a_{1}))\rightarrow
    0$ when $t_{k}\rightarrow 0$. Thus, by passing  to subsequences,
    both $s_{i}(a_{1})$ and $s_{j}(a_{1})$ converge to
    the same point  $b_{1}\in X$ under the Gromov-Hausdorff convergence of
   $(M, \tilde{\omega}_{t_{k}}) $  to
   $(X, d_{X}) $.  We then define
  $\phi(a_{1})=b_{1}$.
 For $a_{2}$,
     by repeating the above procedure, we obtain
  that a subsequence $s_{i_{j}}(a_{j})$, $j=1,2$,
  converges to  $b_{j}\in X$, $j=1,2$, respectively. Define $\phi(a_{2})=b_{2}$.  By
  repeating this procedure and with a diagonal argument, we can find a
  subsequence of $(M, \tilde{\omega}_{t_{k}})$, denoted by $(M, \tilde{\omega}_{t_{k}})$ also, such
  that  $s_{i_{j}}(a_{j})$
  converges to  $b_{j}\in X$ along the Gromov-Hausdorff convergence. For any
  $a_{j} \in A\cap K$, define $\phi(a_{j})=b_{j}$.

  Now, we prove that  $\phi: A\cap \Int K
\rightarrow X$ is injective.   If it is not true, there are $y_{1}$,
$y_{2}\in A\cap \Int K$ such that $y_{1}\neq y_{2} $, and
$\phi(y_{1})=\phi(y_{2})$, which implies
$d_{\tilde{\omega}_{t_{k}}}(s_{i_{1}}(y_{1}),
s_{i_{2}}(y_{2}))\rightarrow 0 $.  If $\gamma_{k}$ is a minimal
geodesic in $(M, \tilde{\omega}_{t_{k}}) $ connecting
$s_{i_{1}}(y_{1})$  and $s_{i_{2}}(y_{2})$, then
$$C^{-1}{\rm length}_{\omega_{N}}(f(\gamma_{k})\cap
K)\leq {\rm length}_{\tilde{\omega}_{t_{k}}}(\gamma_{k}\cap
f^{-1}(K))\leq d_{\tilde{\omega}_{t_{k}}}(s_{i_{1}}(y_{1}),
s_{i_{2}}(y_{2})), $$ by \eqref{c2} for a constant $C>0$ independent
of $k$. Thus, if $f(\gamma_{k})\subset  K$ for $t_{k}\ll 1$,
  $$d_{\omega_{N}}(y_{1}, y_{2})\leq C{\rm
  length}_{\omega_{N}}(f(\gamma_{k})) \rightarrow 0,$$ or, if $f(\gamma_{k})\cap   N\backslash K$ are not empty
   by passing
  to  a subsequence,  $$d_{\omega_{N}}(y_{1}, \partial K)+d_{\omega_{N}}(\partial K,  y_{2})\leq C{\rm
  length}_{\omega_{N}}(f(\gamma_{k})\cap K) \rightarrow 0.  $$
  In both cases, we obtain contradictions. Thus $\phi: A\cap \Int K
\rightarrow X$ is injective.

  Note that   there is a $r>0$ such that, for any $y\in \Int K$,
  the metric ball $B_{\omega}(y, r)$ is a geodesically  convex set, i.e. for any $ y_{1}$ and $ y_{2}\in
  B_{\omega}(y, r)$, there is a minimal geodesic $\gamma\subset B_{\omega}(y, r)$ connecting $ y_{1}$ and $
  y_{2}$, which implies $$d_{\omega}( y_{1}, y_{2})={\rm
  length}_{\omega}(\gamma)\leq 2r. $$
   We take $r\ll 1$ such that there is a $B_{i}'$ with  $B_{\omega}(y, 2r)\subset B_{i}'
   $. If $y_{1}, y_{2}\in A$, by Proposition \ref{conv1},
\[\begin{split}
d_{X}(\phi(y_{1}),\phi(y_{2}))&=\lim_{t_{k}\rightarrow
   0}d_{\tilde{\omega}_{t_{k}}}(s_{i}(y_{1}),s_{i}(y_{2}))\leq \lim_{t_{k}\rightarrow
   0}{\rm
  length}_{\tilde{\omega}_{t_{k}}}(s_{i}(\gamma))\\
&={\rm
  length}_{\omega}(\gamma)=d_{\omega}( y_{1}, y_{2}).\end{split}\]
  If  $\gamma_{k}$ is   a minimal
  geodesic in $(M, \tilde{\omega}_{t_{k}})$ connecting
  $s_{i}(y_{1}) $ and $s_{i}(y_{2})$, then \eqref{conseq2} implies that
 $$e^{-\frac{\ve(t_k)}{2}}{\rm
  length}_{\omega}(f(\gamma_{k})\cap B_{\omega}(y, 2r))\leq {\rm
  length}_{\tilde{\omega}_{t_{k}}}(\gamma_{k})\rightarrow d_{X}(\phi(y_{1}),\phi(
   y_{2})),$$
for some function $\ve(t)\to 0$ as $t\to 0$.
 If $f(\gamma_{k})\subset  B_{\omega}(y, 2r)$ for
   $t_{k}\ll 1$ by passing to a subsequence, $${\rm
  length}_{\omega}(f(\gamma_{k}))\geq {\rm
  length}_{\omega}(\gamma),$$ since $\gamma$ is a minimal geodesic
  in $(N_{0},\omega) $. If $f(\gamma_{k})\cap N_{0}\backslash B_{\omega}(y,
  2r)$ is not empty for $t_{k}\ll 1$, then there is a $\bar{y}\in f(\gamma_{k})\cap N_{0}\backslash B_{\omega}(y,
  2r)$. Since $ y_{1}$, $ y_{2}\in
  B_{\omega}(y, r)$ and $f(\gamma_{k})$ connects  $ y_{1}$ and  $
  y_{2}$,
   $${\rm
  length}_{\omega}(f(\gamma_{k})\cap B_{\omega}(y, 2r))\geq d_{\omega}(y_{1}, \bar{y})
  +d_{\omega}(y_{2}, \bar{y})\geq  2r \geq {\rm
  length}_{\omega}(\gamma).$$
In both cases,
\[\begin{split}
d_{\omega}( y_{1}, y_{2})&= {\rm
  length}_{\omega}(\gamma) \leq \lim_{t_{k}\rightarrow
   0}{\rm
  length}_{\omega}(f(\gamma_{k})\cap B_{\omega}(y, 2r))\\
&\leq d_{X}(\phi(y_{1}),\phi(
   y_{2})).\end{split}\]
 Thus $d_{\omega}( y_{1}, y_{2})=d_{X}(\phi(y_{1}),\phi(
   y_{2})),$  which shows that $\phi: (A\cap  \Int K, d_{\omega})\rightarrow (X, d_{X})$ is a
   local isometric embedding.  If $\{y_{1,j}\}$ and $\{y_{2,j}\}$
   are two sequences in $A\cap  \Int  K$ such that $\lim_{j\rightarrow\infty}d_{\omega}( y_{i,j},
   y)=0$ for $i=1,2$, then $\lim_{j\rightarrow\infty}d_{\omega}( y_{1,j},
   y_{2,j})=0$ and $\{y_{1,j}, y_{2,j}\}\subset B_{\omega}(y, r)$ for $j\gg 1$.
   Hence $d_{\omega}( y_{1,j}, y_{2,j})=d_{X}(\phi(y_{1,j}),\phi(
   y_{2,j}))$ and $d_{\omega}( y_{i,j}, y_{i,j+\ell})=d_{X}(\phi(y_{i,j}),\phi(
   y_{i,j+\ell}))$ for $j\gg 1$ and any $\ell\geq 0$, which implies that
  $\{\phi(y_{1,j})\}$ and $\{\phi(y_{2,j})\}$
   are two Cauchy  sequences, and converge to a unique point $x\in X$.
By defining  $\phi(y)=x$,   $\phi$ extends
   to a unique  map, denoted still  by $\phi$,  from $\Int K$ to $X$ which is also a local isometric embedding.

   Now  we prove that  $\phi ( \Int K)$ is an open subset of $X$. Let
   $x\in \phi ( \Int K)$, i.e. there is a $y\in \Int K $
   such that $\phi (y)=x $, and let $x'\in X$ with $d_{X}(x,x')<\rho $
   for a constant $\rho < \frac{1}{8}d_{\omega}(y, \partial K)$. From the
   above construction,   $y\in B_{i}'$ for a $B_{i}'$, and $s_{i}(y) \rightarrow x$ under Gromov-Hausdorff
   convergence.
     There is a
   sequence of points $p_{k}\in (M, \tilde{\omega}_{t_{k}})$ such
   that $p_{k} \rightarrow x'$ under the Gromov-Hausdorff
   convergence. If  $\gamma_{k}'$ is a  minimal geodesic connecting
     $s_{i}(y)$ and $p_{k}$ in $(M, \tilde{\omega}_{t_{k}})$, then
      $$d_{\tilde{\omega}_{t_{k}}}(s_{i}(y),p_{k})={\rm
      length}_{\tilde{\omega}_{t_{k}}}(\gamma_{k}')\rightarrow
      d_{X}(x,x').$$ Equation \eqref{conseq2} implies that, for $k\gg 1$,   \[\begin{split}
\frac{1}{2}{\rm
  length}_{\omega}(f(\gamma_{k}')\cap K) &\leq e^{-\frac{\ve(t_k)}{2}}{\rm
  length}_{\omega}(f(\gamma_{k}')\cap K) \\ & \leq {\rm
  length}_{\tilde{\omega}_{t_{k}}}(\gamma_{k}')<2\rho< \frac{1}{4}d_{\omega}(y, \partial K).
\end{split}\]
  Thus $f(p_{k})\in K'\subset \Int K$ where $K'$ is a compact
  subset of $\Int K$.  By passing to a subsequence,
  $f(p_{k})\rightarrow y'$ in $(K', \omega)$.  By Proposition \ref{conv1},
   $d_{\tilde{\omega}_{t_{k}}}(p_{k},
   s_{i_{k}}(f(p_{k})))\rightarrow 0$ when $t_{k}\rightarrow 0$,
   and, thus, $s_{i_{k}}(f(p_{k}))\rightarrow x'$ under the Gromov-Hausdorff
   convergence.  The above construction shows that $\phi(y')=x'$,
   which implies that $\{x'|d_{X}(x,x')<\rho  \}\subset  \phi ( \Int
   K)$. Hence $\phi ( \Int K)$ is open, and $\phi : \Int K \rightarrow \phi ( \Int K) $ is a
   homeomorphism.

   Let $K_0\subset \cdots \subset K_{j}\subset K_{j+1} \subset \cdots \subset N_{0}
   $ be a family of compact subsets with $
   N_{0}=\bigcup_{j}\Int K_{j}$.  Given each $K_{j}$, the above
   argument constructs  a
   local isometric embedding $ \phi_{j}: (\Int K_{j},\omega)
   \rightarrow (X, d_{X})$, which is a homeomorphism onto the image $\phi_{j} (\Int K_{j}) $.
   By the same argument as above, $ \phi_{j}$
   extends to a  local isometric embedding $ \phi_{j+1}: ( \Int K_{j+1},\omega)
   \rightarrow (X, d_{X})$, i.e. $\phi_{j+1}|_{\Int K_{j}}=\phi_{j}
   $, which is a homeomorphism onto the image $\phi_{j+1} (\Int  K_{j+1}) $.
    By  a diagonal argument,  we  obtain a local isometry  $ \phi: ( N_{0},\omega)
   \rightarrow (\phi( N_{0}), d_{X})\subset (X, d_{X})$.
\end{proof}

 The above lemma proves the existence of $\phi$  in
Theorem \ref{main2}, and is an analog  of Lemma 4.1 in  \cite{RZ}
for the collapsing case. In the rest of this section, we prove that
$X_{0}=\phi(N_0)$ is dense in $X$.

  Let $\bar{x}\in X_{0}$ and $\bar{p}_{k}\in M$ such that $\bar{p}_{k}\rightarrow \bar{x}$
  under the Gromov-Hausdorff  convergence of $(M, \tilde{\omega}_{t_{k}})$ to $ (X, d_{X})
  $, and let $$\underline{V}_{k}(p, r)=\frac{{\rm Vol}_{\tilde{\omega}_{t_{k}}}
  (B_{\tilde{\omega}_{t_{k}}}(p,r))}{{\rm Vol}_{\tilde{\omega}_{t_{k}}}
  (B_{\tilde{\omega}_{t_{k}}}(\bar{p}_{k},1))},$$ for any $p\in M$ and $r>0$.   By Theorem 1.6 in
  \cite{CC1},  there is a continuous function
  $\underline{V}_{0}:X\times [0,\infty)\rightarrow [0,\infty)$
  such that,  if $p_{k}\rightarrow x$ under the  convergence of $(M, \tilde{\omega}_{t_{k}})$ to $ (X, d_{X})
  $, then    \begin{equation}\label{e4.1} \underline{V}_{k}(p_{k}, r)\rightarrow \underline{V}_{0}(x,r).
  \end{equation} By Theorem 1.10 in
  \cite{CC1}, $\underline{V}_{0}$ induces a unique Radon
  measure $\nu$ on $X$ such that \begin{equation}\label{e4.2}\nu(B_{d_{X}}(x,r))=\underline{V}_{0}(x,r),
  \ \ \ {\rm and} \ \ \frac{\nu(B_{d_{X}}(x,r_{1}))}{\nu(B_{d_{X}}(x,r_{2}))}\geq \mu
  (r_{1},r_{2})>0, \end{equation} for any  $x\in X$, $r_{1}\leq r_{2}$, where $\mu
  (r_{1},r_{2})$ is a function of $r_{1}$ and $ r_{2}$.  For any
  compact subset $K\subset X$, $$\nu(K)=\lim_{\delta\rightarrow 0}
  \nu_{\delta}(K)=\lim_{\delta\rightarrow 0} \inf\left\{\sum_{i}\underline{V}_{0}(x_{i},r_{i})|
 r_{i}<\delta\right\},$$ where
$\bigcup_{i}B_{d_{X}}(x_{i},r_{i})\supset
  K$.   By scaling $\tilde{\omega}_{t}$ and
$\omega$ by one positive number, we assume that
$B_{\omega}(\phi^{-1}(\bar{x}),2)\subset N_{0}$ and is a  geodesically
convex set. 

  \begin{lemma}\label{l4.3} There is a constant  $\upsilon>0$ such that
$$\nu(X)=\upsilon \int_{M}\omega_{M}^{n}, \ \ \ \ \underline{V}_{0}(x,r)=\upsilon
\int_{f^{-1}(B_{\omega}(\phi^{-1}(x),r))}\omega_{M}^{n},$$ whenever
$x\in X_{0}$ and $r\leq 1$ is such that
$B_{\omega}(\phi^{-1}(x),2r)$ is a geodesically convex subset of
$(N_{0}, \omega) $.
\end{lemma}

\begin{proof}    If $p_{k}\rightarrow x$ under
the convergence of $(M, \tilde{\omega}_{t_{k}})$ to $ (X, d_{X})
  $, we claim that a subsequence of $p_{k}$ converges to a point
  $p'\in f^{-1}(\phi^{-1}(x))$ under the metric $\omega_{M}$ on $M$.
By Lemma  \ref{p00.2},    there is a compact  neighborhood $B\subset
N_{0}$ of $\phi^{-1}(x) $ and a
  section $s:B\rightarrow f^{-1}(B)$  such
  that $s(\phi^{-1}(x))\rightarrow x$ under the Gromov-Hausdorff
  convergence of $(M, \tilde{\omega}_{t_{k}}) $ to $ (X, d_{X}) $.
  Thus $d_{\ti{\omega}_{t_{k}}} (p_{k},s(\phi^{-1}(x)))\rightarrow
   0$ when $t_{k}\rightarrow 0$. By Lemma
   \ref{c2es}, there
   are curves $\gamma_{k}$ connecting $p_{k}$ and $s(\phi^{-1}(x))$ such
   that ${\rm length}_{\ti{\omega}_{t_{k}}}
   (\gamma_{k})=d_{\ti{\omega}_{t_{k}}}(p_{k},s(\phi^{-1}(x)))$,
   and
   $${\rm length}_{\omega_{0}} (f(\gamma_{k})\cap B)={\rm length}_{f^{*}\omega_{0}}
   (\gamma_{k}\cap f^{-1}(B))\leq C^{\frac{1}{2}}
   {\rm length}_{\ti{\omega}_{t_{k}}}
   (\gamma_{k})\rightarrow0.$$ For a $k\gg 1$, if there is a $y_{k}\in f(\gamma_{k})\backslash
   B$, then $$ {\rm length}_{\omega_{0}} (f(\gamma_{k})\cap B)\geq
   d_{\omega_{0}}(y_{k},\phi^{-1}(x))\geq
   \rho,$$ where $\rho>0$ such that $B_{\omega_{0}}(\phi^{-1}(x),
   \rho)\subset B$, which is a contradiction.
  Thus $f(\gamma_{k})\subset B$ for $k\gg 1$, ${\rm
length}_{\omega_{0}} (f(\gamma_{k}))\rightarrow 0$
    and $f(p_{k}) $ converges to $\phi^{-1}(x)$ under the metric $\omega_{0}$.
    By passing to  a subsequence, $p_{k}$ converges to a point
  $p'$ under the metric $\omega_{M}$.
  Since $f^{*}\omega_{0}\leq C'\omega_{M}$ for a constant $C'>0$,
   $d_{\omega_{0}}(f(p_{k}),f(p'))\leq C'^{\frac{1}{2}}
   d_{\omega_{M}}(p_{k},p')\rightarrow0.$  Hence $f(p')=\phi^{-1}(x) $ and $p'\in
   f^{-1}(\phi^{-1}(x))$.

Let $r$ satisfy $r\leq 1$, and $B_{\omega}(\phi^{-1}(x),2r)$
is a geodesically convex subset of $(N_{0}, \omega) $.
  If  $q\in f^{-1}(B_{\omega}(\phi^{-1}(x),2r))$, there is a
  curve
  $\bar{\gamma}$  connecting $p'$ and $q$ such that
  $f(\bar{\gamma})$ is the unique minimal geodesic connecting $\phi^{-1}(x)$ and $f(q)$.  Thanks to \eqref{conseq} we have
 $$f^{*}\omega -\ve(t_k)\omega_M
    \leq  \tilde{\omega}_{t_k} \leq f^{*}\omega +\ve(t_k)\omega_M$$ where
$\ve(t_{k})\rightarrow 0$ when $t_{k}\rightarrow 0$, on
$f^{-1}(B_{\omega}(\phi^{-1}(x),2r))$. We obtain that
\[\begin{split}
d_{\ti{\omega}_{t_{k}}}(p',q)&\leq {\rm
length}_{\ti{\omega}_{t_{k}}}(\bar{\gamma}) \leq {\rm
length}_{\omega}(f(\bar{\gamma}))+C\ve(t_{k})^{\frac{1}{2}}\\
&=d_{\omega}(\phi^{-1}(x),f(q))+C\ve(t_{k})^{\frac{1}{2}}.
\end{split}\]
If $\bar{\gamma}_{k}$ is a minimal geodesic of $\ti{\omega}_{t_{k}}$
connecting $p'$ and $q$, then \eqref{conseq2} gives
$$d_{\ti{\omega}_{t_{k}}}(p',q)={\rm
length}_{\ti{\omega}_{t_{k}}}(\bar{\gamma}_{k})\geq
e^{-\frac{\ve(t_k)}{2}}{\rm
length}_{\omega}(f(\bar{\gamma}_{k})\cap
B_{\omega}(\phi^{-1}(x),2r)).$$
 If $f(\bar{\gamma}_{k})\subset B_{\omega}(\phi^{-1}(x),2r)$, then $${\rm
length}_{\omega}(f(\bar{\gamma}_{k})\cap
B_{\omega}(\phi^{-1}(x),2r))\geq {\rm
length}_{\omega}(f(\bar{\gamma}))= d_{\omega}(\phi^{-1}(x),f(q)),$$
and, otherwise,
 $${\rm
length}_{\omega}(f(\bar{\gamma}_{k})\cap
B_{\omega}(\phi^{-1}(x),2r))\geq 2 r \geq {\rm
length}_{\omega}(f(\bar{\gamma}))= d_{\omega}(\phi^{-1}(x),f(q)),$$
by the same argument as in the proof of Lemma \ref{p00.2}.    Thus
$$e^{-\frac{\ve(t_k)}{2}}d_{\omega}(\phi^{-1}(x),f(q))\leq
d_{\ti{\omega}_{t_{k}}}(p',q) \leq d_{\omega}(\phi^{-1}(x),f(q)) +
C\ve(t_{k})^\frac{1}{2}$$ where $C$ is a constant independent of
$t_{k}$, $p'$ and $q$. Of course if $k$ is large we will have that
$$d_{\omega}(\phi^{-1}(x),f(q))-C\ve(t_k)^{\frac{1}{2}}\leq e^{-\frac{\ve(t_k)}{2}}d_{\omega}(\phi^{-1}(x),f(q)).$$
 Thanks to \eqref{c2}, there is constant
$C>0$ independent of $t_{k}$ such that $\ti{\omega}_{t_{k}}\leq
C\omega_{M}$ on $f^{-1}(B_{\omega}(\phi^{-1}(x),2r))$.  Let
$\gamma'_{k}$ be  minimal geodesics of $\omega_{M} $ connecting
$p_{k}$ and $p'$, which satisfy $\gamma'_{k}\subset
f^{-1}(B_{\omega}(\phi^{-1}(x),2r))$ for $k\gg1$.
 Thus  $$d_{\ti{\omega}_{t_{k}}}(p',p_{k})\leq {\rm
length}_{\ti{\omega}_{t_{k}}}(\gamma'_{k})\leq
C^{\frac{1}{2}}{\rm length}_{\omega_{M}}(\gamma'_{k})=
C^{\frac{1}{2}}d_{\omega_{M}}(p',p_{k})\rightarrow0.$$  The
 triangle  inequality shows that $$
|d_{\ti{\omega}_{t_{k}}}(p_{k},q)-d_{\omega}(\phi^{-1}(x),f(q))
| \leq
C\ve(t_{k})^{\frac{1}{2}}+C^{\frac{1}{2}}d_{\omega_{M}}(p',p_{k}).$$
Hence there is a function  $\rho(t_{k})$ of $t_{k}$ such that
$\rho(t_{k})\rightarrow 0$ when $t_{k}\rightarrow 0$, and
$$f^{-1}(B_{\omega}(\phi^{-1}(x),r-\rho(t_{k})))\subset
B_{\tilde{\omega}_{t_{k}}}(p_{k},r) \subset
f^{-1}(B_{\omega}(\phi^{-1}(x),r+\rho(t_{k}))).$$ We obtain that
$$\lim_{t_{k}\rightarrow 0}\int_{B_{\tilde{\omega}_{t_{k}}}(p_{k},r)}\omega_{M}^{n}
=\int_{f^{-1}(B_{\omega}(\phi^{-1}(x),r))}\omega_{M}^{n}.$$

Note that
$$\tilde{\omega}_{t_{k}}^{n}=c_{t_{k}}t_{k}^{n-m}\omega_{M}^{n}.$$
 Hence
\[\begin{split}
\underline{V}_{k}(p_{k}, r)&=\frac{{\rm Vol}_{\tilde{\omega}_{t_{k}}}
  (B_{\tilde{\omega}_{t_{k}}}(p_{k},r))}{{\rm Vol}_{\tilde{\omega}_{t_{k}}}
(B_{\tilde{\omega}_{t_{k}}}(\bar{p}_{k},1))}\\
&=\frac{\int_{B_{\tilde{\omega}_{t_{k}}}(p_{k},r)}
c_{t_{k}}t_{k}^{n-m}\omega_{M}^{n}}{\int_{B_{\tilde{\omega}_{t_{k}}}(\bar{p}_{k},1)}
c_{t_{k}}t_{k}^{n-m}\omega_{M}^{n}}\to\frac{\int_{f^{-1}(B_{\omega}(\phi^{-1}(x),r))}
 \omega_{M}^{n}}{\int_{f^{-1}(B_{\omega}(\phi^{-1}(\bar{x}),1))}
  \omega_{M}^{n}},
\end{split}\]
 when $t_{k}\rightarrow 0$. By (\ref{e4.1}),
  $$\underline{V}_{0}(x,r)=\upsilon \int_{f^{-1}(B_{\omega}(\phi^{-1}(x),r))}
 \omega_{M}^{n}, \ \ {\rm  where } \ \ \upsilon=\left(\int_{f^{-1}(B_{\omega}(\phi^{-1}(\bar{x}),1))}
  \omega_{M}^{n}\right)^{-1}.$$  Recall the diameter bound from \cite{To,ZT} $$ {\rm
  diam}_{\tilde{\omega}_{t_{k}}}(M) \leq D$$ for a constant $D>0$.
Using  (\ref{e4.1}), we have
$$\nu(X)=\underline{V}_{0}(x,D)=\lim_{t_{k}\rightarrow 0}\underline{V}_{k}(p_{k}, D)
  =\upsilon \int_{M}
 \omega_{M}^{n}.$$
\end{proof}

\begin{proof}[Proof of Theorem \ref{main2}]We prove that
$X_{0}\subset X$ is  dense. If this is not true,  there is a metric
ball $B_{d_{X}}(x',\rho)\subset X\backslash X_{0}$. Note that $$
{\rm
  diam}_{d_{X}}(X) =\lim_{t_{k}\rightarrow 0}{\rm
  diam}_{\tilde{\omega}_{t_{k}}}(M)\leq D.$$ Because of
  (\ref{e4.2}), we have  $$\nu(B_{d_{X}}(x',\rho))\geq \mu
  (\rho,D)\nu(X)=\varpi >0.  $$ For any compact subset $K\subset
  X_{0}$, $$\nu(K)\leq \nu(X)-\varpi=\upsilon \int_{M}\omega_{M}^{n}-\varpi
  $$ by Lemma \ref{l4.3}.   If $B_{d_{X}}(x_{i},r_{i})$ is a family
  of metric balls in $(X,d_{X})$ such that  $r_{i}<\delta\ll 1$,
  $B_{d_{X}}(x_{i},2r_{i})$ is a geodesically  convex subset of $
  X_{0}$, and $\bigcup_{i}B_{d_{X}}(x_{i},r_{i})\supset
  K$,  then
  $$\sum_{i}\underline{V}_{0}(x_{i},r_{i})=\sum_{i}\upsilon
  \int_{f^{-1}(\phi^{-1}(B_{d_{X}}(x_{i},r_{i})))}\omega_{M}^{n}\geq \upsilon
  \int_{f^{-1}(\phi^{-1}(K))}\omega_{M}^{n}$$ by Lemma \ref{l4.3}.    Thus
  $$\upsilon \int_{f^{-1}(\phi^{-1}(K))}\omega_{M}^{n}\leq \lim_{\delta\rightarrow 0}
  \nu_{\delta}(K)=\lim_{\delta\rightarrow 0} \inf\left\{\sum_{i}\underline{V}_{0}(x_{i},r_{i})|
 r_{i}<\delta\right\}= \nu(K).$$ By taking $K$ large enough such that
  $$\nu(K)\geq \upsilon \int_{f^{-1}(N_{0})}\omega_{M}^{n}-\frac{\varpi}{2}=
  \upsilon \int_{M}\omega_{M}^{n}-\frac{\varpi}{2},$$
we obtain a contradiction.
\end{proof}

\begin{remark}In fact, the same proof shows that $\nu(X\backslash X_0)=0$.
\end{remark}


\begin{thebibliography}{99}
\bibitem{AC} E. Amerik, F. Campana, {\em Fibrations m\'eromorphes sur certaines vari\'et\'es \`a fibr\'e canonique trivial}, Pure Appl. Math. Q. {\bf 4} (2008), no. 2, part 1, 509--545.
\bibitem{An1} M.T. Anderson, {\em The $L^{2}$ structure of moduli spaces of
Einstein metrics on 4-manifolds}, Geom. Funct.  Anal. {\bf 2}  (1992), no. 1, 29--89.
\bibitem{LB} C. Birkenhake, H. Lange, {\em Complex abelian varieties. Second edition}, Grundlehren der Mathematischen Wissenschaften,
302. Springer-Verlag, Berlin, 2004. xii+635 pp.
\bibitem{COP} F. Campana, K. Oguiso, T. Peternell, {\em Non-algebraic hyperk\"ahler manifolds}, J. Differential Geom. {\bf 85} (2010), no. 3, 397--424.
\bibitem{CC1}  J. Cheeger, T.H. Colding, {\em On the structure of spaces with Ricci curvature bounded below. I},
 J. Differential Geom. {\bf 46}  (1997),  no. 3, 406--480.
\bibitem{CG1}  J. Cheeger, M. Gromov,  {\em  Collapsing Riemannian manifolds while keeping their curvature bound I},  J. Differential Geom. {\bf 23} (1986), no.3, 309--364.
\bibitem{CT} J. Cheeger, G. Tian, {\em Curvature and injectivity radius estimates for Einstein 4-manifolds}, J. Amer. Math. Soc. {\bf 19} (2006), no. 2, 487--525.
\bibitem{CN} T.H. Colding, A. Naber, {\em Sharp H\"older continuity of tangent cones for spaces with a lower Ricci curvature bound and applications}, arXiv:1102.5003.
\bibitem{DP} J.P. Demailly, N. Pali,  {\em  Degenerate complex Monge-Amp\`{e}re equations over compact K\"{a}hler manifolds}, Internat. J. Math. {\bf 21} (2010), no. 3, 357--405.
\bibitem{EGZ2} P. Eyssidieux,  V. Guedj, A. Zeriahi, {\em A priori $L^{\infty}$-estimates for degenerate
complex Monge-Amp\`{e}re equations}, Int. Math. Res. Not. IMRN {\bf 2008}, Art. ID rnn 070, 8 pp.
\bibitem{Fi} J. Fine, {\em Fibrations with constant scalar curvature K\"ahler metrics and the CM-line bundle}, Math. Res. Lett. {\bf 14} (2007), no. 2, 239--247.
\bibitem{Fu}   K. Fukaya,  {\em Hausdorff convergence of Riemannian manifolds and
its applications}, in {\em Recent topics in differential and analytic geometry}, 143--238, Adv. Stud. Pure Math., 18-I, Academic Press 1990.
\bibitem{GT} D. Gilbarg,  N.S. Trudinger, {\em Elliptic partial differential
equations of second}, Springer 1983.
\bibitem{GSVY} B. Greene, A. Shapere, C. Vafa, S.-T. Yau, {\em Stringy cosmic strings and noncompact Calabi-Yau manifolds}, Nuclear Phys. B {\bf 337} (1990), no. 1, 1--36.
\bibitem{G1} M. Gromov, {\em Metric structures for Riemannian and
non-Riemannian spaces}, Birkh\"{a}user 1999.
\bibitem{GHJ} M. Gross, D. Huybrechts and D. Joyce, \emph{Calabi-Yau
manifolds and related geometries,} Springer-Verlag 2003.
\bibitem{GS} M. Gross, B. Siebert, {\em Mirror symmetry via logarithmic degeneration data, II}, J. Algebraic Geom. {\bf 19} (2010), no. 4, 679--780.
\bibitem{GW} M. Gross, P.M.H. Wilson, \emph{Large complex structure limits of $K3$ surfaces}, J. Differential Geom. {\bf 55} (2000), no. 3, 475--546.
\bibitem{HT} B. Hassett, Y. Tschinkel, {\em Rational curves on holomorphic symplectic fourfolds}, Geom. Funct. Anal. {\bf 11} (2001), no. 6, 1201--1228.
\bibitem{H} H.-J. Hein, {\em Gravitational instantons from rational elliptic surfaces}, J. Amer. Math. Soc. {\bf 25} (2012), no. 2, 355--393.
\bibitem{Huy1} D. Huybrechts, \emph{Compact hyper-K\"ahler manifolds:
basic results}, Invent. Math. {\bf 135} (1999), 63--113; Erratum, Invent. Math. {\bf 152} (2003), 209--212.
\bibitem{Hw} J.-M. Hwang, \emph{Base manifolds for fibrations of projective irreducible symplectic manifolds}, Invent. Math. {\bf 174} (2008), 625--644.
\bibitem{KSp} K. Kodaira, D.C. Spencer, {\em On deformations of complex analytic structures. I, II}, Ann. of Math. (2) {\bf 67} (1958), 328--466.
\bibitem{Ko} S. Ko\l odziej, {\em  The complex Monge-Amp\`{e}re
equation}, Acta Math. {\bf 180} (1998), no. 1, 69--117.
\bibitem{KS} M. Kontsevich, Y. Soibelman, {\em  Homological mirror symmetry and torus fibrations}, in \it Symplectic geometry and mirror symmetry, \rm 203--263, World Sci. Publishing 2001.
\bibitem{Mats} D. Matsushita, \emph{On fibre space structures of a projective
irreducible symplectic manifold,} Topology {\bf 38} (1999), 79--83.
\bibitem{NT} A. Naber, G. Tian, {\em Geometric structures of collapsing Riemannian manifolds I}, in {\em Surveys in geometric analysis and relativity}, Higher Education Press, Beijing, 2011.
\bibitem{Ro} X. Rong, {\em Convergence and collapsing theorems in Riemannian
geometry}, Handbook of geometric analysis, No. 2, 193--299, Adv. Lect. Math. (ALM), 13, Int. Press 2010.
\bibitem{RoZ} X. Rong, Y.G. Zhang {\em Continuity of extremal transitions and flops for Calabi-Yau manifolds}, J.  Differential Geom. {\bf 89} (2011), no. 2, 233--269.
 \bibitem{RZ} W.D. Ruan, Y.G. Zhang, {\em  Convergence of Calabi-Yau
manifolds},  Adv. Math. {\bf 228} (2011), no. 3, 1543--1589.
\bibitem{Sawon} J. Sawon, \emph{Abelian fibred holomorphic
symplectic fibrations}, Turkish J. Math. {\bf 27} (2003), 197--230.
\bibitem{Si} Y.-T. Siu, \emph{Lectures on Hermitian-Einstein metrics for stable bundles and K\"ahler-Einstein metrics}, DMV Seminar, 8. Birkh\"auser Verlag, Basel, 1987.
\bibitem{SSW} J. Song, G.Sz\'ekelyhidi, B. Weinkove, {\em The K\"ahler-Ricci flow on projective bundles}, Int. Math. Res. Not. {\bf 2012}.
\bibitem{ST} J. Song, G. Tian, {\em The K\"ahler-Ricci flow on surfaces of positive Kodaira dimension}, Invent. Math. {\bf 170} (2007), 609--653.
\bibitem{ST2} J. Song, G. Tian, {\em Canonical measures and the K\"ahler-Ricci flow}, J. Amer. Math. Soc. {\bf 25} (2012), no. 2, 303--353.
\bibitem{SYZ} A. Strominger, S.-T. Yau, E. Zaslow, {\em Mirror symmetry is $T$-duality}, Nuclear Phys. B {\bf 479} (1996), no. 1-2, 243--259.
\bibitem{To} V. Tosatti, {\em  Limits of Calabi-Yau metrics when the K\"{a}hler class degenerates}, J. Eur. Math. Soc. (JEMS) {\bf 11} (2009), no.4, 755--776.
\bibitem{To1} V. Tosatti, {\em  Adiabatic limits of Ricci-flat K\"ahler metrics}, J. Differential Geom. {\bf 84} (2010), no. 2, 427--453.
\bibitem{To2} V. Tosatti, {\em Degenerations of Calabi-Yau metrics},
in {\em Geometry and Physics in Cracow,} Acta Phys. Polon. B Proc. Suppl. {\bf 4} (2011), no.3, 495--505.
\bibitem{Verb2} M. Verbitsky, \emph{Mirror symmetry for hyper-K\"ahler manifolds},
Mirror symmetry, III (Montreal, PQ, 1995), 115--156, AMS/IP Stud.
Adv. Math., 10, Amer.\ Math.\ Soc. 1999.
\bibitem{Verb1} M. Verbitsky, \emph{Mapping class group and a global Torelli theorem for hyperk\"ahler manifolds,}, arXiv:0908.4121.
\bibitem{verb3} M. Verbitsky, \emph{HyperK\"ahler SYZ
conjecture and semipositive line bundles,} Geom. Funct.  Anal. {\bf 19}, (2010), no.5, 1481--1493.
\bibitem{We} J. Wehler, {\em Isomorphie von Familien kompakter komplexer Mannigfaltigkeiten}, Math. Ann. {\bf 231} (1977/78), no. 1, 77--90.
\bibitem{W} P.M.H. Wilson, \emph{Metric limits of Calabi-Yau manifolds}, in \emph{The Fano Conference}, 793--804, Univ. Torino, Turin, 2004.
\bibitem{Ya1}  S.-T. Yau,  \emph{On the Ricci curvature of a compact K\"ahler manifold and the complex Monge-Amp\`ere equation, I}, Comm. Pure Appl. Math. {\bf 31} (1978), 339--411.
\bibitem{yau2} S.-T. Yau, \emph{Problem section} in \emph{Seminar on Differential Geometry}, pp. 669--706,
Ann. of Math. Stud. {\bf 102}, Princeton Univ. Press, 1982 (problem 49).
\bibitem{yau3} S.-T. Yau, \emph{Open problems in geometry}, Proc. Sympos. Pure
Math. {\bf 54} (1993), 1--28 (problem 88).
\bibitem{ZT} Y.G. Zhang, {\em Convergence of K\"ahler manifolds and calibrated fibrations}, PhD thesis, Nankai Institute of Mathematics, 2006.
 \end{thebibliography}
\end{document}